\documentclass[review]{elsarticle}

\usepackage{lineno,hyperref}
\usepackage{amsmath,amssymb}
\usepackage{graphicx}
\usepackage{mathtools}
\usepackage{tikz}
\usepackage{subfigure}
\usepackage{mathptmx}
\usepackage{latexsym}
\usepackage{amsthm}
\newtheorem{theorem}{Theorem}[section]
\newtheorem{assumption}{Assumption}[section]
\newtheorem{definition}{Definition}[section]
\newtheorem{remark}{Remark}[section]
\newtheorem{corollary}{Corollary}[section]
\newtheorem{lemma}{Lemma}[section]

\modulolinenumbers[5]

\begin{document}

\begin{frontmatter}
\title{On the limit of a two-phase flow problem in thin porous media domains of Brinkman-type}

\author{Alaa Armiti-Juber}
\address{Institute of Mechanics, Structural Analysis, and Dynamics of Aerospace Structures, University of Stuttgart, Pfaffenwaldring 27, 70569 Stuttgart, Germany}
\ead{alaa.armiti-juber@isd.uni-stuttgart.de}



\begin{abstract}
We study the process of two-phase flow in thin porous media domains of Brinkman-type. This is generally described by a model of coupled, mixed-type differential equations of fluids' saturation and pressure. To reduce the model complexity, different approaches that utilize the thin geometry of the domain have been suggested. 

We focus on a reduced model that is formulated as a single nonlocal evolution equation of saturation. It is derived by applying standard asymptotic analysis to the dimensionless coupled model, however, a rigid mathematical derivation is still lacking. In this paper, we prove that the reduced model is the analytical limit of the coupled two-phase flow model as the geometrical parameter of domain's width--length ratio tends to zero. Precisely, we prove the convergence of weak solutions for the coupled model to a weak solution for the reduced model as the geometrical parameter vanishes.
\end{abstract}

\begin{keyword}
Two-phase flow \sep Brinkman regimes \sep Model reduction in thin domains \sep Mathematical convergence \sep Weak solution
\end{keyword}

\end{frontmatter}


\section{Introduction}
\label{intro}
We study the process of fluid displacement by another fluid in nondeformable saturated porous media domains of thin structure. This is crucial for many environmental and industrial applications. Examples are enhanced oil recovery in oil reservoirs and carbon dioxide sequestration in saline aquifers. Such processes are typically described by the two-phase flow model, which is a coupled system of mixed-type differential equations \cite{Helmig1997}. The complexity of the model and the large volume of such domains in the subsurface lead to high computational complexity. However, different approaches that utilize the thin geometry of these domains have been suggested to reduce the model's complexity. An example is the dimensional reduction approach by vertical integration in the field of petroleum studies \cite{Lake1989}, hydrogeology \cite{Bear1972,Helmig1997}, and carbon dioxide sequestration \cite{Gasda2009,Gasda2011}. Other examples are the asymptotic approach in Darcy \cite{Yortsos} and  Brinkman regimes \cite{Armiti-Juber2018} and the multiscale model approach \cite{Guo2014}. We refer to \cite{Armiti-Juber2018} for a comparative study on the accuracy and efficiency of the asymptotic approach over the others, in addition to an equivalence result with the multiscale approach. A recent approach suggests an adaptive algorithm that couples the dimensional reduction approach with the full model \cite{Becker2018}. It is based on a local criterion that determines the applicability of the reduced model.
\medskip 

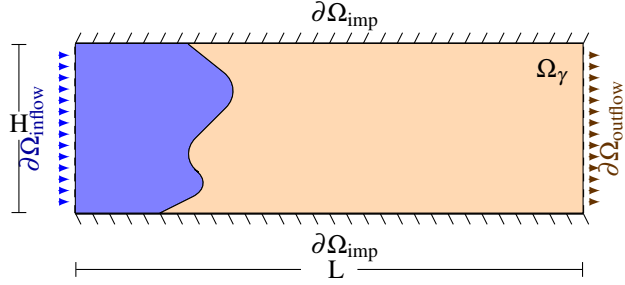
\begin{figure}
\centering
\begin{tikzpicture}[scale = 0.75,>=latex]
\draw[dashed, thick] (0,0) -- coordinate (y axis mid) (0,3);
\node[rotate=90, above=0.3cm, black!40!blue] at (y axis mid) {$\partial\Omega_{\text{inflow}}$};
\node[rotate=90, above=-7.4cm, black!60!orange] at (y axis mid) {$\partial\Omega_{\text{outflow}}$};
\draw[dashed, thick] (9,0) -- coordinate (x axis mid) (9,3);
\draw[very thick] (0,0) -- (9,0);
\draw[very thick] (0,3) -- (9,3); 

\draw[|-] (0,-1)--(4,-1) ;\draw[-|] (5,-1)--(9,-1);
\draw[very thick] (4.3,-1) node[anchor=west] {L};
\draw[|-] (-1,0)--(-1,1.3) ;\draw[-|] (-1,1.8)--(-1,3);
\draw[very thick] (-1,1.25) node[anchor=south] {H};

\draw[->,blue](-0.3,0.2)-- (-0.1,0.2);\draw[->,blue](-0.3,0.4)-- (-0.1,0.4);\draw[->,blue](-0.3,0.6)-- (-0.1,0.6);\draw[->,blue](-0.3,0.8)-- (-0.1,0.8);\draw[->,blue](-0.3,1)-- (-0.1,1);\draw[->,blue](-0.3,1.2)-- (-0.1,1.2);\draw[->,blue](-0.3,1.4)-- (-0.1,1.4);\draw[->,blue](-0.3,1.6)-- (-0.1,1.6);\draw[->,blue](-0.3,1.8)-- (-0.1,1.8);\draw[->,blue](-0.3,2)-- (-0.1,2);\draw[->,blue](-0.3,2.2)-- (-0.1,2.2);\draw[->,blue](-0.3,2.4)-- (-0.1,2.4);\draw[->,blue](-0.3,2.6)-- (-0.1,2.6);\draw[->,blue](-0.3,2.8)-- (-0.1,2.8);

\draw[->,black!60!orange](9.1,0.2)-- (9.3,0.2);\draw[->,black!60!orange](9.1,0.4)-- (9.3,0.4);\draw[->,black!60!orange](9.1,0.6)-- (9.3,0.6);\draw[->,black!60!orange](9.1,0.8)-- (9.3,0.8);\draw[->,black!60!orange](9.1,1)-- (9.3,1);\draw[->,black!60!orange](9.1,1.2)-- (9.3,1.2);\draw[->,black!60!orange](9.1,1.4)-- (9.3,1.4);\draw[->,black!60!orange](9.1,1.6)-- (9.3,1.6);\draw[->,black!60!orange](9.1,1.8)-- (9.3,1.8);\draw[->,black!60!orange](9.1,2)-- (9.3,2);\draw[->,black!60!orange](9.1,2.2)-- (9.3,2.2);\draw[->,black!60!black!60!orange](9.1,2.4)-- (9.3,2.4);\draw[->,black!60!orange](9.1,2.6)-- (9.3,2.6);\draw[->,black!60!orange](9.1,2.8)-- (9.3,2.8);

\draw[-] (0,0) -- (0.1,-0.2);\draw[-] (0.3,0) -- (0.4,-0.2);\draw[-] (0.6,0) -- (0.7,-0.2);\draw[-] (0.9,0) -- (1,-0.2);\draw[-] (1.2,0) -- (1.3,-0.2);\draw[-] (1.5,0) -- (1.6,-0.2);\draw[-] (1.8,0) -- (1.9,-0.2);\draw[-] (2.1,0) -- (2.2,-0.2);\draw[-] (2.4,0) -- (2.5,-0.2);\draw[-] (2.7,0) -- (2.8,-0.2);\draw[-] (3,0) -- (3.1,-0.2);\draw[-] (3.3,0) -- (3.4,-0.2);\draw[-] (3.6,0) -- (3.7,-0.2);\draw[-] (3.9,0) -- (4,-0.2);\draw[-] (4.2,0) -- (4.3,-0.2);\draw[-] (4.5,0) -- (4.6,-0.2);\draw[-] (4.8,0) -- (4.9,-0.2);\draw[-] (5.1,0) -- (5.2,-0.2);\draw[-] (5.4,0) -- (5.5,-0.2);\draw[-] (5.7,0) -- (5.8,-0.2);\draw[-] (6,0) -- (6.1,-0.2);\draw[-] (6.3,0) -- (6.4,-0.2);\draw[-] (6.6,0) -- (6.7,-0.2);\draw[-] (6.9,0) -- (7,-0.2);\draw[-] (7.2,0) -- (7.3,-0.2);\draw[-] (7.5,0) -- (7.6,-0.2);\draw[-] (7.8,0) -- (7.9,-0.2);\draw[-] (8.1,0) -- (8.2,-0.2);\draw[-] (8.4,0) -- (8.5,-0.2);\draw[-] (8.7,0) -- (8.8,-0.2);\draw[-] (9,0) -- (9.1,-0.2);

\draw[-] (0,3) -- (0.1,3.2);\draw[-] (0.3,3) -- (0.4,3.2);\draw[-] (0.6,3) -- (0.7,3.2);\draw[-] (0.9,3) -- (1,3.2);\draw[-] (1.2,3) -- (1.3,3.2);\draw[-] (1.5,3) -- (1.6,3.2);\draw[-] (1.8,3) -- (1.9,3.2);\draw[-] (2.1,3) -- (2.2,3.2);\draw[-] (2.4,3) -- (2.5,3.2);\draw[-] (2.7,3) -- (2.8,3.2);\draw[-] (3,3) -- (3.1,3.2);\draw[-] (3.3,3) -- (3.4,3.2);\draw[-] (3.6,3) -- (3.7,3.2);\draw[-] (3.9,3) -- (4,3.2);\draw[-] (4.2,3) -- (4.3,3.2);\draw[-] (4.5,3) -- (4.6,3.2);\draw[-] (4.8,3) -- (4.9,3.2);\draw[-] (5.1,3) -- (5.2,3.2);\draw[-] (5.4,3) -- (5.5,3.2);\draw[-] (5.7,3) -- (5.8,3.2);\draw[-] (6,3) -- (6.1,3.2);\draw[-] (6.3,3) -- (6.4,3.2);\draw[-] (6.6,3) -- (6.7,3.2);\draw[-] (6.9,3) -- (7,3.2);\draw[-] (7.2,3) -- (7.3,3.2);\draw[-] (7.5,3) -- (7.6,3.2);\draw[-] (7.8,3) -- (7.9,3.2);\draw[-] (8.1,3) -- (8.2,3.2);\draw[-] (8.4,3) -- (8.5,3.2);\draw[-] (8.7,3) -- (8.8,3.2);\draw[-] (9,3) -- (9.1,3.2);

\draw[fill=blue!50] (0,3)-- (0,0) -- (1.5,0)[rounded corners=10pt]--(2.5,0.5)--(1.8,1)--(3,2.2)--(2,3);
\draw[fill=orange!30]  (9,3)--(9,0)-- (1.5,0)[rounded corners=10pt]--(2.5,0.5)--(1.8,1)--(3,2.2) -- (2,3);
\draw[thick] (8,2.5)  node[anchor=west] {$\Omega_{\gamma}$};
\draw[thick] (4,3.5)  node[anchor=west] {$\partial\Omega_{\text{imp}}$};
\draw[thick] (4,-0.6)  node[anchor=west] {$\partial\Omega_{\text{imp}}$};
\end{tikzpicture}
\caption{An illustration of the displacement process of a wetting phase into a thin domain $\Omega_{\gamma}$ \cite{Armiti-Juber2018}. }
\label{fig:Omega}
\end{figure}

In this paper, we explore the limit of the two-phase flow model in porous media domains of Brinkman type as the width--length ratio of the domain tends to zero. We prove the analytical convergence of weak solutions for this model to weak solutions for the reduced model resulting from the asymptotic approach in \cite{Armiti-Juber2018}. In fact, it is shown in \cite{Armiti-Juber2018} that numerical solutions of the two-phase flow model converge to those of the reduced model as the domain's geometrical ratio tends to zero. Similar results on the analytical convergence of mathematical models have been established for other different applications. Examples are the convergence of the two-phase flow model as the viscosity of one of the phases approaches zero \cite{Henry2012}, the convergence of a mathematical model describing crystal dissolution in thin strips as the thickness vanishes \cite{Duijn04crystaldissolution}, and the convergence of a reactive transport equation in fractured porous media as the thickness of the fractures tends to zero \cite{pop17fracture}.

We consider the homogenized flow of two incompressible immiscible fluids in the rectangular domain $\Omega_{\gamma}=(0,L)\times(0,H)$ such that $H\ll L$ (Figure \ref{fig:Omega}), where $\gamma\coloneqq H/L$ is the geometrical parameter. Using dimensionless variables, governing equations for such flows are given by the so-called Brinkman two-phase flow model (BTP-model),
\begin{align}
\begin{array}{rl}
\partial_{t}S-\beta_1 \partial_{txx} S-\beta_2\partial_{tzz} S+\partial_{x}\big(f(S)U\big)+ \partial_{z}\big(f(S)Q\big)&=0,\\
U&=-\lambda_{tot}(S) \partial_{x} p,\\
 \gamma^2  Q &=-\lambda_{tot}(S ) \partial_{z} p ,\\  
  \partial_{x}U +\partial_{z}Q  &=0
\end{array}
\label{eq:BTP}
\end{align}
in the dimensionless domain $\Omega\times(0,T)$, where $\Omega=(0,1)\times(0,1)$ and $T>0$ \cite{Armiti-Juber2018}. We refer to \ref{sec:dimensionless} for details on the derivation of this model. The unknowns here are the saturation $S =S (x,z,t) \in [0,1]$ of the wetting (or invading) phase and the global pressure $p =p (x,z,t)\in \mathbb{R}$. The component $U =U (x,z,t)\in\mathbb{R}$, for any $(x,z,t)\in \Omega\times(0,T)$ is the horizontal velocity and $Q =Q (x,z,t)\in\mathbb{R}$ is the vertical one. The total mobility function $\lambda_{tot}=\lambda_{tot}(S ) \in (0,\infty)$ is the mobility sum of both phases. We refer to \cite{Helmig1997} for possible choices for the mobilities. The function $f=f(S )\in[0,1]$ is the given fractional flow function, which is determined using the fluids' mobilities and viscosities. The parameters $\beta_1=\mu_e/L^2$ and $\beta_2=\mu_e/H^2$, where $\mu_e$ is the effective viscosity, determine the flow regime. The case $\mu_e=0$ results in the so-called Darcy regime, while $\mu_e>0$ is referred to as the Brinkman regime \cite{Helmig1997}.

The reduced model resulting from the asymptotic approach in Brinkman regimes is derived in \cite{Armiti-Juber2018}. This is done by applying standard asymptotic analysis, in terms of the geometrical parameter $\gamma$, to the dimensionless BTP-model \eqref{eq:BTP}. In the limit, this leads to a pressure function independent of the vertical coordinate, a result that is usually called the vertical equilibrium, see e.g. \cite{Guo2014,Yortsos}. This result is then used to reformulate the velocity components in the reduced model as nonlocal operators of saturation. For details on the derivation of this model we refer to \ref{sec:asymptotic}. The reduced model is a nonlocal nonlinear evolution equation of saturation. It is given as
\begin{eqnarray}
\partial_{t}S +\partial_{x}\Bigl(f(S)U[S]\Bigr)+\partial_{z}\Bigl(f(S)Q[S]\Bigr) -\beta_1 \partial_{txx}S-\beta_2 \partial_{tzz}S=0
\label{eq:BVE}
\end{eqnarray}
in $\Omega\times(0,T)$, with the velocity components
\begin{equation}
U[S]=\dfrac{\hat{U}_{\text{inflow}}\lambda_{tot}(S)}{\int_{0}^{1}\lambda_{tot}(S) dz},\quad \quad Q[S]= -\partial_{x}\int_{0}^{z}U[S(\cdot,r,\cdot)]dr.
\label{eq:velocity}
\end{equation}
Here, $\hat{U}_{\text{inflow}}=\hat{U}_{\text{inflow}}(t)$ is the vertically averaged horizontal velocity at the left boundary of the domain. This inflow velocity can be evaluated using a one-dimensional elliptic equation of the vertically averaged pressure $\hat{p}=\int_0^1 p(.,z,.)\,dz$. However, it is set in \cite{Yortsos} to be a constant $\hat{U}_{\text{inflow}}=1$. In \cite{Armiti-Juber2018} it is eliminated from the model by rescaling the time $t$ variable using $t\mapsto\bar{t}=\int_0^t \hat{U}_{\text{inflow}}(r)dr+\hat{U}_{\text{inflow}}(0)t$. In addition, the definition of the velocity components $U$ and $Q$ in \eqref{eq:velocity} still fulfills the incompressibility constraint
\begin{align}
\partial_x U + \partial_z Q=0.
\label{eq:incompressible}
\end{align} 

\medskip

The equations \eqref{eq:BVE} and \eqref{eq:velocity} are called here as in \cite{Armiti-Juber2018}, the Brinkman Vertical Equilibrium model (BVE-model). This model is a proper reduction of the full BTP-model \eqref{eq:BTP} in thin domains as it describes the vertical dynamics in the domain. Moreover, it is computationally more efficient than the full mixed BTP-model for saturation and global pressure (see \cite{Armiti-Juber2018}). This is a consequence of the velocity equations in \eqref{eq:velocity} computed from saturation directly, without solving an elliptic equation for the global pressure as in full BTP-model \eqref{eq:BTP}.
\medskip

The main goal of this paper is a rigid mathematical derivation of the reduced BVE-model \eqref{eq:BVE} and \eqref{eq:velocity} from the full BTP-model \eqref{eq:BTP}. We do this by showing that the reduced model is the analytical limit of the full BTP-model in domains with vanishing width--length ratio $\gamma$. This paper is structured as follows. In section \ref{sec:Preliminaries} we choose the initial and boundary conditions that fit to the two-phase displacement process.  Then, we give the definitions of weak solutions for the full BTP-model \eqref{eq:BTP} and the reduced BVE-model \eqref{eq:BVE} and \eqref{eq:velocity}. After that, section \ref{sec:a priori} proves a set a priori estimates on a sequence of weak solutions for the BTP-model. These are essential to prove the convergence of the sequence in section \ref{sec:convergence} as the ratio $\gamma$ approaches zero. Section \ref{sec:numerical} presents an example that shows the numerical convergence of full BTP-model to the reduced BVE-model as $\gamma$ vanishes. Section \ref{sec:conclusion} concludes the paper. Finally, the derivation of the dimensionless BTP-model \eqref{eq:BTP} is summarized in \ref{sec:dimensionless}, while the derivation of the BVE-model \eqref{eq:BVE}, \eqref{eq:velocity} using the asymptotic approach as in \cite{Armiti-Juber2018} is summarized in \ref{sec:asymptotic}.

\section{Preliminaries}
\label{sec:Preliminaries}
In this section, we give the initial and boundary conditions associated with the displacement process in the dimensionless domain $\Omega$. Then we provide the definition of weak solution for the BTP-model and the BVE-model.

The BTP-model is closed with the initial and boundary conditions 
\begin{align}
\begin{array}{rll}
 S(\cdot,\cdot,0)&=S^0 \quad\quad &\text{ in } \Omega, \\
 S &= S_{\text{inflow}}\quad\quad &\text{ on } \partial \Omega_{\text{inflow}}\times[0,T], \\
 S &= 0 \quad\quad &\text{ on } \partial \Omega_{\text{imp}}\cup\partial\Omega_{\text{outflow}}\times[0,T], \\
 \nabla p \cdot \textbf{n} & = q \quad\quad &\text{ on } \partial \Omega\times[0,T], \\
 \int_{\Omega} p(x,z,t)\,dx\,dz & =0 \quad\quad &\text{ on }t\in(0,T),\\
 p & =p_D \quad\quad & \text{ on } \partial \Omega_{\text{imp}}\times[0,T], \\
 Q & =0 \quad\quad & \text{ on } \partial \Omega_{\text{imp}}\times[0,T], 
 \end{array}
 \label{eq:IBC-BTP}
\end{align}
where $S_{\text{inflow}}=S_{\text{inflow}}(z)$ and $q=q(x,z,t)$ are given functions and $p_D$ is a constant. Note that $\partial \Omega_{\text{inflow}}=\{0\}\times (0,1)$ is the inflow boundary, $\Omega_{\text{outflow}}=\{1\}\times(0,1)$ is the outflow boundary, and $\Omega_{\text{imp}}=(0,1)\times\{0,1\}$ corresponds to the impermeable lower and upper boundaries (Figure \ref{fig:Omega}). We also use the notations $\Omega_T=\Omega\times (0,T)$ and impose the following assumptions. 
\begin{assumption}
\label{ass:Bexistence}
\begin{enumerate}
 \item The bounded domain $\Omega\subset\mathbb{R}^2$ has a Lipschitz continuous boundary $\partial \Omega$ and $0<T<\infty$.
 \item The inflow saturation $S_{\text{inflow}}$ is bounded in $L^{\infty}((0,1)\times(0,T))$. 
 \item We require $S^0\in H^1(\Omega)$ and $S^0=S_{\text{inflow}}$ at $\partial\Omega_{\text{inflow}}$.
 \item The pressure function $q$ satisfies $q\in L^2(\partial\Omega\times(0,T))$.
 \item The fractional flow function $f\in C^1((0,1))$ is Lipschitz continuous, bounded, nonnegative, monotone increasing and $f(0)=0$, such that there exist numbers $M,\,L>0$ with $f\leq M,\,f'\leq L$.
 \item The total mobility function $\lambda_{tot}\in C^1((0,1))$ is Lipschitz continuous, bounded and strictly positive, such that there exist numbers $a,\, M,\,L>0$ with $0<a<\lambda_{tot}\leq M$ and $\lvert\lambda_{tot}'\rvert\leq L$.
\end{enumerate} 
\end{assumption}

\begin{remark}
 Note that the velocity components at the boundaries of the domain can be evaluated using the velocity equations in \eqref{eq:BTP} and the boundary conditions on saturation and pressure in \eqref{eq:IBC-BTP}. For example, we define the velocity $U_{\text{inflow}}=U_{\text{inflow}}(z,t)$ at the inflow boundary as
 \begin{align}
 U_{\text{inflow}}=\lambda_{tot}(S_{\text{inflow}}) q|_{\partial\Omega_{\text{inflow}}},
  \label{eq:Uinflow}
 \end{align}
 and 
 the velocity $U_{\text{outflow}}=U_{\text{outflow}}(z,t)$ at the outflow boundary as
 \begin{align}
 U_{\text{outflow}}=\lambda_{tot}(0) q|_{\partial\Omega_{\text{outflow}}}.
  \label{eq:Uoutflow}
 \end{align}
Using Assumption \ref{ass:Bexistence}(4) and \ref{ass:Bexistence}(6), we have $U_{\text{inflow}},\, U_{\text{outflow}}\in L^2((0,1)\times(0,T))$. In addition, the constant pressure $p_D$ at the boundary $\partial\Omega_{\text{imp}}$ leads to 
\begin{align}
 \begin{array}{cll}
  U & = 0 &\quad \text{on} \quad \partial\Omega_{\text{imp}}.
 \end{array}
 \label{eq:U-imp}
\end{align}
\label{rem:U-boundary}
\end{remark}

\begin{definition}
 For any $\gamma>0$, we call $(S^{\gamma}, p^{\gamma}, U^{\gamma}, Q^{\gamma})$ a weak solution of the BTP-model \eqref{eq:BTP} with the initial and boundary conditions \eqref{eq:IBC-BTP} if 
 \begin{enumerate}
  \item $S^{\gamma} \in H^1(0,T;H^1(\Omega))$, $p^{\gamma} \in L^2(0,T;H^1(\Omega))$, and $U^{\gamma},\,Q^{\gamma}\in L^2(\Omega\times(0,T))$ with
  \begin{align}
   \int_{0}^{T}\int_{\Omega} &\big(\partial_{t}S^{\gamma} \phi  - f(S^{\gamma}) U^{\gamma}\partial_x \phi - f(S^{\gamma})Q^{\gamma}\partial_z\phi \big) \,dx\,dz\,dt\nonumber\\& +\beta_1\int_{0}^{T}\int_{\Omega} \partial_{tx} S^{\gamma}\partial_{x}\phi \,dx\,dz\,dt+\beta_2\int_{0}^{T}\int_{\Omega} \partial_{tz} S^{\gamma}\partial_{z}\phi \,dx\,dz\,dt\nonumber\\&=\int_0^T\int_{\partial\Omega_{\text{inflow}}}f(S^{\gamma}_{\text{inflow}})U^{\gamma}_{\text{inflow}}\phi(0,z,t)\,dz\,dt,
    \label{eq:weakmodel}
 \end{align}
 for any test function $\phi \in L^2(0,T;C^0(\Omega))$. 
 \item The velocity components satisfy
 \begin{align}
  \int_{\Omega} U^{\gamma} \psi \,dx\,dz =- \int_{\Omega} \lambda_{tot}(S^{\gamma})\partial_x p^{\gamma} \psi \,dx\,dz,
  \label{eq:weakU}
 \end{align}
 and
 \begin{align}
 \gamma^2 \int_{\Omega} Q^{\gamma} \psi \,dx\,dz = -\int_{\Omega} \lambda_{tot}(S^{\gamma})\partial_z p^{\gamma} \psi \,dx\,dz,
   \label{eq:weakW}
 \end{align}
 for any test function $\psi \in L^2(\Omega)$ and almost everywhere in $(0,T)$.
 \item The following two weak incompressibility relations
\begin{align}
 \int_{\Omega} \lambda_{tot}(S^{\gamma})\partial_x p^{\gamma}\partial_x \theta +  \frac{1}{\gamma^2}  \lambda_{tot}(S^{\gamma})\partial_z p^{\gamma} \partial_z\theta\,dx\,dz=\int_{\partial \Omega}\lambda_{tot}(S^{\gamma})q \theta \,d\sigma,
 \label{eq:weakincomm-gamma}
\end{align}
and 
\begin{align}
\int_{\Omega} \big( U^{\gamma}\partial_x \theta +& Q^{\gamma}\partial_z\theta\big) \,dx\,dz=-\int_{\partial \Omega_{\text{inflow}}} U_{\text{inflow}}\theta(0,z) \,dz,
 \label{eq:weakincomm}
\end{align}
hold for any test function $\theta \in C^0(\Omega)$ and almost everywhere in $(0,T)$, with $\theta(1,z)=0$.
\item $S^{\gamma}(.,.,0)=S^0$ almost everywhere in $\Omega$.
 \end{enumerate}
\label{def:BTP-weaksolution}
\end{definition}

\begin{remark}
 Definition \ref{def:BTP-weaksolution} implies that weak solutions for the BTP-model \eqref{eq:BTP} satisfy
\begin{align}
 S^{\gamma} \in C([0,T];H^1(\Omega)).
 \label{eq:continuity}
\end{align}
\label{rem:time-continuity}
\end{remark}

\begin{definition}
A function $S\in H^1(0,T;H^1(\Omega))$ is called a weak solution of the BVE-model \eqref{eq:BVE}, \eqref{eq:velocity} and \eqref{eq:incompressible} with the initial and boundary conditions \eqref{eq:IBC-BTP} whenever the following conditions are fulfilled,
\begin{enumerate}
 \item $U[S],\,Q[S] \in L^2(\Omega_T)$ and 
 \begin{align}
  \int_{0}^{T}\int_{\Omega} \big(\partial_{t}S \phi-& f(S) U[S]\partial_x \phi - f(S)Q[S]\partial_z\phi \big) \,dx\,dz\,dt\nonumber\\ +&\beta_1\int_{0}^{T}\int_{\Omega}  \partial_{tx}S \partial_x\phi \,dx\,dz\,dt+\beta_2\int_{0}^{T}\int_{\Omega}  \partial_{tz}S \partial_z\phi \,dx\,dz\,dt\nonumber\\&=\int_0^T\int_{\partial\Omega_{\text{inflow}}}f(S_{\text{inflow}})U_{\text{inflow}}\phi(0,z,t)\,dz\,dt,
 \label{eq:weaksat-BVE}
\end{align}
holds for all test functions $\phi\in L^{2}(0,T;C^{0}(\Omega))$.
\item The velocity components satisfy
\begin{align}
 \int_{\Omega} U \psi \,dx\,dz = \int_{\Omega} \frac{\hat{U}_{\text{inflow}}\lambda_{tot}(S)}{\int_{0}^{1}\lambda_{tot}(S) dz} \psi \,dx\,dz,
 \label{eq:weakU-BVE}
 \end{align}
 
 \begin{align}
 \int_{\Omega} Q \psi \,dx\,dz = -\int_{\Omega} \partial_{x}\int_{0}^{z}U[S(\cdot,r,\cdot)]dr \psi \,dx\,dz,
  \label{eq:weakW-BVE}
 \end{align}
for any $\psi \in L^2(\Omega)$ and almost everywhere in $(0,T)$.

\item The weak incompressibility property 
\begin{align}
       \int_{\Omega} \big( U[S]\partial_x \theta &+ Q[S]\partial_z\theta\big) \,dx\,dz\,dt=-\int_{\partial \Omega_{\text{inflow}}} U_{\text{inflow}}\theta(0,z) \,dz,
       \label{eq:weakincomp-limit}
      \end{align}
holds for all test functions $\theta\in C^{0}(\Omega)$ and almost everywhere in time, with $\theta(1,z)=0$.
\item $S(.,.,0)=S^0$ almost everywhere in $\Omega$.
\end{enumerate}
\label{def:BVE-weaksolution}
\end{definition}

\section{A priori Estimates}
\label{sec:a priori}

In the following, we prove a set of a priori estimates on the components of the sequence of weak solutions $\{(S^{\gamma},p^{\gamma},U^{\gamma},Q^{\gamma})\}_{\gamma>0}$ for the BTP-model \eqref{eq:BTP}. These are essential for the convergence analysis as $\gamma$ tends to zero in the next section. Note that the existence of weak solutions for the BTP-model is proved in \cite{Coclite2014}, while for the BVE-model is proved in \cite{Armiti-Juber2019}. 

\begin{lemma}
Let $\{(S^{\gamma},p^{\gamma},U^{\gamma},Q^{\gamma})\}_{\gamma>0}$ be a sequence of weak solutions for the BTP-model \eqref{eq:BTP}. If Assumption \ref{ass:Bexistence} holds, then the sequence $\{S^{\gamma}\}_{\gamma>0}$ satisfies the estimate
 \begin{align*}
 \sup_{t\in[0,T]} \Big(\Vert S^{\gamma}(t)\Vert_{L^2(\Omega)}& + \beta_1 \Vert \partial_x S^{\gamma}(t)\Vert_{L^2(\Omega)}+ \beta_2 \Vert \partial_z S^{\gamma}(t)\Vert_{L^2(\Omega)}\Big)\\&\leq  \Vert S^0\Vert^2_{L^2(\Omega)}+\beta_1 \Vert \partial_x S^0\Vert^2_{L^2(\Omega)}+\beta_2 \Vert \partial_z S^0\Vert^2_{L^2(\Omega)}+C_{\text{inflow}},
 \end{align*}
 where $C_{\text{inflow}}$ is a constant depending on the data at the inflow boundary only.
 \label{lem:apriori1}
\end{lemma}
\begin{proof}
We Choose the test function $\phi=S^{\gamma}\chi_{[0,t)}$ in equation \eqref{eq:weakmodel}, where $\chi_{[0,t)}$ is the characteristic function and $t\in (0,T]$ is arbitrary. Then, we obtain
 \begin{align}
   \int_{0}^{t}\int_{\Omega} &\big(\partial_{t}S^{\gamma} S^{\gamma}  - f(S^{\gamma}) U^{\gamma}\partial_x S^{\gamma} - f(S^{\gamma})Q^{\gamma}\partial_z S^{\gamma} \big) \,dx\,dz\,dt\nonumber\\& +\beta_1\int_{0}^{t}\int_{\Omega} \partial_{tx} S^{\gamma}\partial_{x} S^{\gamma} \,dx\,dz\,dt+\beta_2\int_{0}^{t}\int_{\Omega} \partial_{tz} S^{\gamma}\partial_{z} S^{\gamma} \,dx\,dz\,dt\nonumber\\&=\int_0^t\int_{\partial\Omega_{\text{inflow}}}f(S^{\gamma}_{\text{inflow}})U^{\gamma}_{\text{inflow}}S^{\gamma}_{\text{inflow}}\,dz\,dt.
  \label{eq:aprioriess1-1}
 \end{align}
Using the incompressibility relation \eqref{eq:weakincomm} and the boundary condition on the outflow boundary, the second and third terms on the left side of the above equation satisfy,
 \begin{align}
   \int_{0}^{t}\int_{\Omega} f(S^{\gamma}) \textbf{V}^{\gamma}\cdot \nabla S^{\gamma}\,dx\,dz\,dt =&  \int_{0}^{t}\int_{\Omega} \textbf{V}^{\gamma}\cdot \nabla  F(S^{\gamma})\,dx\,dz\,dt, \nonumber\\ =& -\int_0^t\int_{\partial\Omega_{\text{inflow}}} U_{\text{inflow}} F(S_{\text{inflow}})\,dz\,dt,
   \label{eq:aprioriess1-2}
 \end{align}
where $\textbf{V}^{\gamma}=(U^{\gamma}, Q^{\gamma})^T$ and $F(S)= \int_0^{S} f(q)dq$. Integrating the first term on the left side of \eqref{eq:aprioriess1-1} and using the time-continuity of $S^{\gamma}$ in Remark \ref{rem:time-continuity}, we obtain
\begin{align}
 \int_{0}^{t}\int_{\Omega} \bigl(\partial_{t}S^{\gamma} S^{\gamma}\,dx\,dz\,dt = \frac{1}{2}\int_{0}^{t}\int_{\Omega} \partial_{t}(S^{\gamma})^2\,dx\,dz\,dt= \frac{1}{2}\big( \Vert S^{\gamma}(t) \Vert^2_{L^2(\Omega)} - \Vert S^0 \Vert^2_{L^2(\Omega)} \big).
 \label{eq:aprioriess1-3}
\end{align}
In the same way, we have
\begin{align}
\beta_1\int_{0}^{t}\int_{\Omega} \partial_{tx} S^{\gamma}\partial_{x} S^{\gamma} \,dx\,dz\,dt=\frac{\beta_1}{2}\big( \Vert \partial_x S^{\gamma}(t) \Vert^2_{L^2(\Omega)} - \Vert \partial_x S^0 \Vert^2_{L^2(\Omega)} \big),
 \label{eq:aprioriess1-4}
\end{align}
and 
\begin{align}
 \beta_2\int_{0}^{t}\int_{\Omega} \partial_{tz} S^{\gamma}\partial_{z} S^{\gamma} \,dx\,dz\,dt=\frac{\beta_2}{2}\big( \Vert \partial_z S^{\gamma}(t) \Vert^2_{L^2(\Omega)} - \Vert \partial_z S^0 \Vert^2_{L^2(\Omega)} \big).
 \label{eq:aprioriess1-5}
\end{align}
Substituting equations \eqref{eq:aprioriess1-2}-\eqref{eq:aprioriess1-5} into \eqref{eq:aprioriess1-1} yields
 \begin{align}
 &\sup_{t\in[0,T]} \Big(\frac{1}{2}\Vert S^{\gamma}(t)\Vert^2_{L^2(\Omega)} +\frac{\beta_1}{2} \Vert \partial_x S^{\gamma}(t)\Vert^2_{L^2(\Omega)}+ \frac{\beta_2}{2} \Vert \partial_z S^{\gamma}(t)\Vert^2_{L^2(\Omega)}\Big)\nonumber \\& + \int_0^t\int_{\partial\Omega_{\text{inflow}}}U_{\text{inflow}}F(S_{\text{inflow}})\,dz\,dt = \frac{1}{2}\Vert S^0\Vert^2_{L^2(\Omega)}+\frac{\beta_1}{2} \Vert \partial_x S^0\Vert^2_{L^2(\Omega)} +\frac{\beta_2}{2} \Vert \partial_z S^0\Vert^2_{L^2(\Omega)}\nonumber\\& + \int_0^t\int_{\partial\Omega_{\text{inflow}}}U_{\text{inflow}}f(S_{\text{inflow}})S_{\text{inflow}}\,dz\,dt.
 \label{eq:estimate1-2}
 \end{align}
 The boundedness of $S_{\text{inflow}}$, $U_{\text{inflow}}$, and $f$ by Assumption \ref{ass:Bexistence}(2), \ref{ass:Bexistence}(5), and\ref{ass:Bexistence}(6), respectively, implies
  \begin{align*}
 \sup_{t\in[0,T]} \Big(\Vert S^{\gamma}(t)\Vert_{L^2(\Omega)}& + \beta_1 \Vert \partial_x S^{\gamma}(t)\Vert_{L^2(\Omega)}+ \beta_2 \Vert \partial_z S^{\gamma}(t)\Vert_{L^2(\Omega)}\Big)\\&\leq  \Vert S^0\Vert^2_{L^2(\Omega)}+\beta_1 \Vert \partial_x S^0\Vert^2_{L^2(\Omega)}+\beta_2 \Vert \partial_z S^0\Vert^2_{L^2(\Omega)}+C_{\text{inflow}},
 \end{align*}
 where $C_{\text{inflow}}=2M\Vert U_{\text{inflow}} \Vert_{L^{\infty}(\partial\Omega_{\text{inflow}}\times(0,T))}\Vert S_{\text{inflow}} \Vert_{L^{\infty}(\partial\Omega_{\text{inflow}}\times(0,T))}$. 

\end{proof}

The following lemma proves an estimate on the sequence of pressure's gradient. In the limit $\gamma\rightarrow 0$, the estimate is equivalent to the vertical equilibrium assumption (see e.g. \cite{Guo2014}. It is also essential to formulate the limit pressure as an operator of saturation.
\begin{lemma}
Let $\{(S^{\gamma},p^{\gamma},U^{\gamma},Q^{\gamma})\}_{\gamma>0}$ be a sequence of weak solutions for the BTP-model \eqref{eq:BTP}. If Assumption \ref{ass:Bexistence} holds, then there exists a constant $c>0$, independent of the parameter $\gamma$, such that the sequence $\{p^{\gamma}\}_{\gamma>0}$ satisfies the estimate
\begin{align*}
   (1-\gamma^2)\Vert \partial_z p^{\gamma}\Vert^2_{L^2(\Omega)} + \gamma^2  \Vert \partial_x p^{\gamma}\Vert^2_{L^2(\Omega)}\leq \frac{2c M^2\gamma^2 }{a^2} \Vert q \Vert^2_{L^2(\partial \Omega)}.
 \end{align*}
    \label{lem:apriori2}
\end{lemma}
\begin{proof}
 We choose the test function $\theta=p^{\gamma}$ in equation \eqref{eq:weakincomm-gamma}, then we have
 \begin{align*}
 \int_{\Omega} \lambda_{tot}(S^{\gamma})\Big(\big(\partial_x p^{\gamma}\big)^2 +  \frac{1}{\gamma^2}  \big(\partial_z p^{\gamma})^2\Big)\,dx\,dz = & \int_{\partial \Omega}\lambda_{tot}(S^{\gamma})q p^{\gamma} \,d\sigma.
 \end{align*}
 Using Assumption \ref{ass:Bexistence}(6) on the total mobility then applying Cauchy's inequality to the right side yields
 \begin{align*}
 a \Vert\partial_x p^{\gamma} \Vert^2_{L^2(\Omega)} + \frac{a}{\gamma^2} \Vert \partial_z p^{\gamma}\Vert^2_{L^2(\Omega)}\leq \,& \frac{M^2}{2\epsilon}\int_{\partial \Omega} q^2 \,d\sigma + \frac{\epsilon}{2}\int_{\partial\Omega} (p^{\gamma})^2\,d\sigma,
 \end{align*}
for any constant $\epsilon>0$. Applying the Trace theorem to the second term on the right side, then using Poincar\'e's inequality with the zero mean condition on the pressure (see the boundary conditions in \eqref{eq:IBC-BTP}) produces
  \begin{align*}
   a \Vert\partial_x p^{\gamma} \Vert^2_{L^2(\Omega)} + \frac{a}{\gamma^2}  \Vert \partial_z p^{\gamma}\Vert^2_{L^2(\Omega)}\leq \,& \frac{M^2}{2\epsilon}\int_{\partial \Omega} q^2 \,d\sigma +\frac{c\epsilon}{2}\Vert \nabla p^{\gamma} \Vert_{L^2(\Omega)},
 \end{align*}
where $c>0$ is a constant resulting from the above two Sobolev embedding theorems. Choosing $\epsilon=\frac{a}{c}$ and noting that $\gamma<1$, yields
\begin{align*}
 \frac{a}{2}\gamma^2\Vert\partial_x p^{\gamma} \Vert^2_{L^2(\Omega)} + \frac{a}{2}(1-\gamma^2) \Vert \partial_z p^{\gamma}\Vert^2_{L^2(\Omega)}\leq \,& \frac{cM^2}{2a}\gamma^2\int_{\partial \Omega} q^2 \,d\sigma .
\end{align*}
This simplifies to
\begin{align*}
 \gamma^2\Vert\partial_x p^{\gamma} \Vert^2_{L^2(\Omega)} + (1-\gamma^2) \Vert \partial_z p^{\gamma}\Vert^2_{L^2(\Omega)}\leq \,& \frac{c M^2\gamma^2}{a^2}\int_{\partial \Omega} q^2 \,d\sigma ,
\end{align*}
which is the required estimate.
\end{proof}

\begin{corollary}
If Assumption \ref{ass:Bexistence} holds, then there exists a constant $C>0$, independent of the parameter $\gamma$, such that the velocity components $U^{\gamma}$ and $ W^{\gamma}$ satisfy 
\begin{align*}
 \|U^{\gamma}\|_{L^2(\Omega_T)}&\leq C\Vert q \Vert^2_{L^2(\partial \Omega_T)}, \\ \|Q^{\gamma}\|_{L^2(\Omega_T)}&\leq \frac{C}{1-\gamma^2}\Vert q \Vert^2_{L^2(\partial \Omega_T)}.
\end{align*}
\label{cor:Vestimate}
\end{corollary}
\begin{proof}
 The definition of $U^{\gamma}$ and Lemma \ref{lem:apriori2} implies that
 \begin{align*}
  \|U^{\gamma}\|_{L^2(\Omega_T)}= \int_0^T\int_{\Omega} |\lambda_{tot}(S^{\gamma})\partial_x p^{\gamma}|
  \leq  M^2\|\partial_x p^{\gamma}\|_{L^2(\Omega_T)}\leq C\Vert q \Vert^2_{L^2(\partial \Omega_T)},
 \end{align*}
where $C= \tfrac{c M^4}{a^2}$. Similarly, the component $Q^{\gamma}$ satisfies
 \begin{align*}
  \gamma^2\|Q^{\gamma}\|_{L^2(\Omega_T)}= \int_0^T\int_{\Omega} |\lambda_{tot}(S^{\gamma})\partial_z p^{\gamma}|&
  \leq  M^2\|\partial_z p^{\gamma}\|_{L^2(\Omega_T)}\\&\leq \frac{c M^4\gamma^2}{a^2(1-\gamma^2)}\Vert q \Vert^2_{L^2(\partial \Omega_T)}.
 \end{align*}
Hence, we have
 \begin{align*}
  \|Q^{\gamma}\|_{L^2(\Omega_T)}\leq \frac{C}{1-\gamma^2}\Vert q \Vert^2_{L^2(\partial \Omega_T)}.
 \end{align*}
\end{proof}

In the following lemma we prove an estimate on the time-partial derivative of the weak solution $S^{\gamma}$ and its derivative $\partial_x S^{\gamma}$.

\begin{lemma}
Let $\{(S^{\gamma},p^{\gamma},U^{\gamma},Q^{\gamma})\}_{\gamma>0}$ be a sequence of weak solutions for the BTP-model \eqref{eq:BTP}. If Assumption \ref{ass:Bexistence} holds, then there exists a constant $C>0$, independent of the parameters $\gamma$ and $\mu_e$, such that the sequence $\{S^{\gamma}\}_{\gamma>0}$ satisfies the estimate
\begin{align*}
 \Vert\partial_t S^{\gamma}\Vert_{L^2(\Omega_T)} + \frac{3\beta_1}{4} \Vert \partial_{tx}  S^{\gamma}\Vert_{L^2(\Omega_T)}&+ \frac{3\beta_2}{4} \Vert  \partial_{tz} S^{\gamma}\Vert_{L^2(\Omega_T)} \\&\leq \frac{M^2}{\mu_e}\Big(C+\frac{C}{1-\gamma^2} \Big)\Vert q \Vert^2_{L^2(\partial \Omega_T)}.
\end{align*}
\label{lem:apriori3}
\end{lemma}
\begin{proof}
 We consider the weak formulation \eqref{eq:weakmodel} in Definition \ref{def:BTP-weaksolution} with the test function $\phi=\partial_t S^{\gamma}$. Then, using Cauchy's inequality we obtain
 \begin{align*}
  &\int_{0}^{T}\int_{\Omega} \big(\partial_{t}S^{\gamma}\big)^2+\beta_1( \partial_{tx}S^{\gamma})^2 + \beta_2(\partial_{tz}S^{\gamma})^2\,dx\,dz\,dt \\= &\int_{0}^{T}\int_{\Omega}  f(S^{\gamma}) \big( U^{\gamma}\partial_{xt}S^{\gamma} + Q^{\gamma}\partial_{zt}S^{\gamma} \big) \,dx\,dz\,dt,\\
  \leq & ~\frac{1}{\beta_1}\|f(S^{\gamma})U^{\gamma}\|^2_{L^2(\Omega)}+ \frac{\beta_1}{4}\|\partial_{tx}S^{\gamma}\|^2_{L^2(\Omega)}+\frac{1}{\beta_2}\|f(S^{\gamma})Q^{\gamma}\|^2_{L^2(\Omega)}+ \frac{\beta_2}{4}\|\partial_{tz}S^{\gamma}\|^2_{L^2(\Omega)}.
  \end{align*}
  Note that the term on the inflow boundary vanished as a result of the time-independent choice for the inflow saturation $S_{\text{inflow}}$. This reduces to 
  \begin{align}
   \Vert\partial_t S^{\gamma}\Vert_{L^2(\Omega_T)}& + \frac{3\beta_1}{4} \Vert  \partial_{tx} S^{\gamma}\Vert_{L^2(\Omega_T)}+ \frac{3\beta_2}{4} \Vert  \partial_{tz} S^{\gamma}\Vert_{L^2(\Omega_T)}\nonumber \\&\leq \frac{1}{\beta_1}\|f(S^{\gamma})U^{\gamma}\|^2_{L^2(\Omega)}+\frac{1}{\beta_2}\|f(S^{\gamma})Q^{\gamma}\|^2_{L^2(\Omega)} .
   \label{eq:S-tz-estimate}
  \end{align}
Now, using Corollary \ref{cor:Vestimate}, we obtain
\begin{align*}
   \Vert\partial_t S^{\gamma}\Vert_{L^2(\Omega_T)}& + \frac{3\beta_1}{4} \Vert  \partial_{tx} S^{\gamma}\Vert_{L^2(\Omega_T)}+ \frac{3\beta_2}{4} \Vert  \partial_{tz} S^{\gamma}\Vert_{L^2(\Omega_T)} \\&\leq \frac{M^2}{\beta}\Big(C+\frac{C}{1-\gamma^2} \Big)\Vert q \Vert^2_{L^2(\partial \Omega_T)},
  \end{align*}
   where $\beta=\min\{\beta_1,\,\beta_2\}$ and $C>0$ is a constant defined as in Corollary \ref{cor:Vestimate}.
\end{proof}

\section{Convergence Analysis}
\label{sec:convergence}
In this section we prove the analytical convergence of the sequence of weak solutions $\{(S^{\gamma},p^{\gamma},U^{\gamma},Q^{\gamma})\}_{\gamma>0}$ for the BTP-model \eqref{eq:BTP} to a weak solution of the BVE-model \eqref{eq:BVE}, \eqref{eq:velocity} as the geometrical parameter $\gamma$ tends to $0$. The main result of paper is summarized in this theorem.

\begin{theorem}
Let $\{(S^{\gamma},p^{\gamma},U^{\gamma},Q^{\gamma})\}_{\gamma>0}$ be a sequence of weak solutions for the BTP-model \eqref{eq:BTP} with the initial and boundary conditions \eqref{eq:IBC-BTP}. If Assumption \ref{ass:Bexistence} holds, then there exists a subsequence of the weak solutions $\{S^{\gamma}, p^{\gamma},U^{\gamma},Q^{\gamma}\}_{\gamma>0}$, denoted in the same way, and functions $S\in H^1(0,T;H^1(\Omega))$, $p\in L^2((0,T);H^1(0,1))$, $U\in L^2(\Omega_T)$ and $ Q\in L^2(\Omega_T)$ such that
 \begin{align*}
 \begin{array}{rll}
  S^{\gamma}&\rightarrow S &\quad\quad \text{ in } L^2(\Omega_T),\\
  \nabla S^{\gamma}&\rightharpoonup \nabla S &\quad\quad \text{ in } H^1(0,T;L^2(\Omega)),\\
   p^{\gamma} &\rightharpoonup p &\quad\quad \text{ in } L^2(0,T;H^1(\Omega)),\\
  U^{\gamma} &\rightharpoonup U&\quad\quad \text{ in } L^2(\Omega_T),\\
  Q^{\gamma} &\rightharpoonup Q &\quad\quad \text{ in } L^2(\Omega_T)
   \end{array}
 \end{align*}
as $\gamma$ tends to zero. Further, the limit pressure $p$ is independent of the $z$ coordinate and satisfies $\partial _x p=-\tfrac{\hat{U}_{\text{inflow}}}{\int_0^1\lambda_{tot}(S)\,dz}$. The functions $S,U,Q$ satisfy the equations \eqref{eq:weaksat-BVE}, \eqref{eq:weakU-BVE} and \eqref{eq:weakW-BVE} in Definition \ref{def:BVE-weaksolution}, respectively.
\label{thm:MainThm}
\end{theorem}
\begin{proof}
 The estimate in Lemma \ref{lem:apriori1} implies the existence of a weakly convergent subsequence of $\{ S^{\gamma}\}_{\gamma>0}$, denoted in the same way, and a function $S\in L^2(0,T;H^1(\Omega))$ with
\begin{align}
\begin{array}{cll}
 S^{\gamma} &\rightharpoonup S &\quad\text{ in } L^2(0,T;H^1(\Omega)),
\label{eq:S-weak}
\end{array}
\end{align}
as $\gamma\rightarrow 0$. In addition, the estimate in Lemma \ref{lem:apriori3} implies
\begin{align}
\begin{array}{cll}
 \nabla S^{\gamma} &\rightharpoonup \nabla S &\quad\text{ in } H^1(0,T;L^2(\Omega)),
\label{eq:Sx-weak}
\end{array}
\end{align}
as $\gamma\rightarrow 0$. The Rellich-Kondrachov compactness theorem and the boundedness of the domain imply the embedding $H^1(0,T; H^1(\Omega)) \Subset L^2(\Omega_T)$. Thus, the weak convergence results \eqref{eq:S-weak} and \eqref{eq:Sx-weak} lead to the strong convergence
\begin{align}
 S^{\gamma} \rightarrow S \in L^2(\Omega_T).
 \label{eq:S-strong}
\end{align}
This strong convergence and the a priori estimate from Lemma \ref{lem:apriori1} imply that the limit $S$ also satisfies
\begin{align}
 S,\,\nabla S \in L^{\infty}(0,T;H^1(\Omega)).
\end{align}
Moreover, we have
\begin{align}
 S\in C([0,T];H^1(\Omega)).
 \label{eq:S-timecontinuous}
\end{align}
The strong convergence result in \eqref{eq:S-strong} and the Lipschitz continuity of $f$ and $\lambda_{tot}$ imply
\begin{align}
\begin{array}{cl}
 f(S^{\gamma})\rightarrow f(S) &\text{ in } L^2(\Omega_T),\\
 \lambda_{tot}(S^{\gamma})\rightarrow \lambda_{tot}(S)&\text{ in } L^2(\Omega_T).
\end{array} 
\label{eq:strongconv-f-lambda}
\end{align}

Now, we consider the estimate in Lemma \ref{lem:apriori2} and let $\gamma\rightarrow 0$. Then, we have
\begin{align}
 \Vert \partial_z p^{\gamma}\Vert^2_{L^2(\Omega)}\rightarrow 0,
\end{align}
as $\gamma\rightarrow 0$. This, consequently, leads to the uniform estimate
\begin{align*}
   \Vert \partial_x p^{\gamma}\Vert^2_{L^2(\Omega)}\leq \frac{c M^2 }{a^2} \Vert q \Vert^2_{L^2(\partial \Omega)}.
 \end{align*}
Hence, there exists a weakly convergent subsequence of $\{ p^{\gamma}\}_{\gamma>0}$, denoted in the same way, and a $z$-independent function $p=p(x)$ with $p\in L^2(0,T; H^1((0,1)))$ such that
\begin{align}
 p^{\gamma} \rightharpoonup p \quad\text{ in } L^2(0,T; H^1(\Omega)).
 \label{eq:limitpressure}
\end{align}
This convergence result corresponds to the vertical equilibrium assumption for almost horizontal flows in thin domains.

The strong convergence of $\lambda_{tot}$ in \eqref{eq:strongconv-f-lambda} and the weak convergence of $p$ in \eqref{eq:limitpressure} imply the weak convergence of $U^{\gamma}=\lambda_{tot}(S^{\gamma})\partial_x p^{\gamma}$ to the limit $U=\lambda_{tot}(S)\partial_x p$ such that
\begin{align}
 U^{\gamma} \rightharpoonup U=\lambda_{tot}(S)\partial_x p \quad\text{ in } L^2(\Omega_T).
\end{align}

Corollary \ref{cor:Vestimate} implies the boundedness of $Q^{\gamma}$ in $L^2(\Omega_T)$. Hence, up to a subsequence, there exists a function $Q\in  L^2(\Omega_T) $ such that
\begin{align}
\int_0^T \int_{\Omega} Q^{\gamma } \phi \,dx\,dz\,dt\rightarrow \int_0^T \int_{\Omega} Q \, \phi \,dx\,dz\,dt,
\label{eq:B-Qweak1}
\end{align}
for any test function $\phi\in L^2(\Omega_T)$. We also have the weak convergence of the products
\begin{align}
\begin{array}{clc}
 f(S^\gamma)U^{\gamma}&\rightharpoonup f(S)U &\text{ in } L^2(\Omega_T),\\
 f(S^\gamma)Q^{\gamma}&\rightharpoonup f(S)Q &\text{ in } L^2(\Omega_T).
\end{array}
\label{eq:conv-product}
\end{align}
All above convergence results imply that equation \eqref{eq:weakmodel} in Definition \ref{def:BTP-weaksolution} converge to
\begin{align}
   \int_{0}^{T}\int_{\Omega} \partial_{t}S \phi-& f(S) U\partial_x \phi - f(S)Q\partial_z\phi +\beta_1 \partial_{tx}S \partial_x\phi +\beta_2 \partial_{tz}S \partial_z\phi \,dx\,dz\,dt\nonumber\\&=\int_0^T\int_{\partial\Omega_{\text{inflow}}}f(S_{\text{inflow}})U_{\text{inflow}}\phi(0,z,t)\,dz\,dt,
    \label{eq:weakmodellimit}
 \end{align}
 for any $\phi\in L^2(0,T;C^0(\Omega))$. Further, the velocity component $U$ satisfies
  \begin{align}
  \int_{\Omega} U \psi \,dx\,dz = -\int_{\Omega} \lambda_{tot}(S)\partial_x p \psi \,dx\,dz,
  \label{eq:weakUlimit}
 \end{align}
 for any $\psi\in L^2(\Omega)$ almost everywhere in $(0,T)$. Also the limit $Q$ satisfies
 \begin{align}
\int_{\Omega} U \partial_x \phi \,dx\,dz + \int_{\Omega} Q\partial_z\phi \,dx\,dz= -  \int_{\partial \Omega_{\text{inflow}}} U_{\text{inflow}}\phi(0,\cdot) \,dz,
 \label{eq:weakincommlimit}
\end{align}
 for any test function $\phi\in C^0(\Omega)$ almost everywhere in $(0,T)$.

In the following, we evaluate the limit pressure $p$ using the limit saturation $S$ and the velocity at the inflow boundary $U_{\text{inflow}}$. This consequently leads to limit velocity operators $U$ and $Q$ that depend on $S$ and $U_{\text{inflow}}$ only. So, we consider equation \eqref{eq:weakincommlimit} with a test function $\phi=\phi(x)$ that satisfies $\phi\in C^0((0,1))$, $\phi(1)=0$, $\phi(0)=1$ and $\int_0^1\phi'\,dx=-1$. Then, equation \eqref{eq:weakincommlimit} reduces to
 \begin{align}
\int_{\Omega} U \phi' \,dx\,dz = -  \int_{\partial \Omega_{\text{inflow}}} U_{\text{inflow}} \,dz.
\label{eq:reduction1}
\end{align}
We also define the vertically-averaged operator
\begin{align}
 \hat U(x,t)=\int_0^1 U(x,z,t)\,dz ,
\end{align}
for almost all $x\in(0,1)$ and $t\in(0,T)$ and choose the test function $\psi=\phi'$. Then, equation \eqref{eq:reduction1} is reformulated as
 \begin{align}
  \int_0^1 \hat U \phi'\,dx=-\hat U_{\text{inflow}}.
  \label{eq:result1}
 \end{align}
Similarly equation \eqref{eq:weakUlimit} with the $z$-independent test function $\psi= \phi'\in L^2((0,1))$ reduces to
 \begin{align}
  \int_{0}^1 \hat U \phi' \,dx = -\int_{0}^1 \partial_x p \hat{\lambda}_{tot}(S)\phi' \,dx,
  \label{eq:reduction2}
 \end{align}
where $\hat{\lambda}_{tot}\coloneqq \int_0^1\lambda_{tot}(S)\,dz$ is the vertically-averaged total mobility. Substituting equation \eqref{eq:result1} into \eqref{eq:reduction2} yields
 \begin{align}
  \hat U_{\text{inflow}}= \int_{0}^1 \partial_x p \hat{\lambda}_{tot}(S)\phi' \,dx,
 \end{align}
 As the limit pressure $p=p(x,t)$ and the vertically averaged mobility $\hat{\lambda}_{tot}(S)$ are independent of the $z$-coordinate, then using $\int_0^1\phi'\,dx=-1$ we obtain 
\begin{align} 
 \partial_x p = -\frac{\hat U_{\text{inflow}}}{\hat{\lambda}_{tot}(S)}.
 \label{eq:pressure-saturation}
\end{align}
Substituting this formula into \eqref{eq:weakUlimit} allows reformulating the horizontal velocity $U$ component as
\begin{align}
 \int_{\Omega}U \psi\,dx\,dz = \int_{\Omega} \frac{\hat U_{\text{inflow}} \lambda_{tot}(S)}{\hat{\lambda}_{tot}(S)} \psi\,dx\,dz,
 \label{eq:limit-U}
\end{align}
for any $\psi\in L^2(\Omega)$ and almost everywhere in $(0,T)$. 

The last step in the proof is to evaluate the limit velocity $Q$. For this, it is necessary first to prove the claim
\begin{align}
 \int_{\Omega} U[S] \partial_x \phi \,dx\,dz = -\int_{\Omega} \int_0^z U[.,r;S]\,dr\partial_{xz}\phi \,dx\,dz,
 \label{eq:claim1}
\end{align}
for any test function $\phi\in H^1(\Omega)$. The proof starts with applying Gauss' theorem to the right side of the equation above together with equation \eqref{eq:U-imp} in Remark \ref{rem:U-boundary}. Then, we have
\begin{align}
 \int_{\Omega} \int_0^z U[.,r;S]\,dr\partial_{xz}\phi \,dx\,dz = -\int_{\Omega} \partial_z \int_0^z U[.,r;S]\,dr\partial_{x}\phi \,dx\,dz.
\end{align}
Using summation by parts, it holds that
\begin{align*}
 &\int_{\Omega} \dfrac{\int_0^z U[S(x,r,t)]\,dr-\int_{-\Delta z}^{z-\Delta z}U[S(x,r,t)]\,dr}{\Delta z}\,\partial_{x}\phi\,dx\,dz\,dt\\ =& \int_{\Omega} \frac{1}{\Delta z}\int_{z-\Delta z}^{z}U[S(x,r,t)]\,dr\,\partial_{x}\phi\,dx\,dz- \int_{\Omega} \frac{1}{\Delta z}\int_{-\Delta z}^{0}U[S(x,r,t)]\,dr\,\partial_{x}\phi\,dx\,dz.
\end{align*}
Letting $\Delta z\rightarrow 0$ and using Lebesgue's Differentiation theorem \cite{Evans} together with equation\eqref{eq:U-imp}, we obtain
\begin{align}
 \int_{\Omega}\partial_z\int_0^z U[S(x,r,t)]\,dr\,\partial_{x}\phi\,dx\,dz\,dt=\int_{\Omega} U[S(x,z,t)]\,\partial_{x}\phi\,dx\,dz,
\label{eq:B-weakincomp2}
\end{align}
for almost all $z\in(0,1)$, which proves the claim. Thus, substituting \eqref{eq:claim1} into the weak incompressibility relation \eqref{eq:weakincommlimit} yields 
\begin{align*}
\int_{\Omega} Q\partial_z\phi\,dx\,dz =\int_{\Omega} \int_0^z U[S(.,r,.)]\,dr\partial_{xz}\phi \,dx\,dz-\int_0^1 U_{\text{inflow}}(r,.)\phi(0,r)\,dr,
\end{align*}
for any test function $\phi\in C^0(\Omega)$ with $\phi(1,z)=0$ for almost all $z\in (0,1)$. We apply again Gauss' theorem to the first term on the right side the equation above and use the choice $\phi(1,z)=0$. Then we have
\begin{align}
\int_{\Omega} Q\partial_z\phi\,dx\,dz =&-\int_{\Omega} \partial_x \int_0^z U[S(.,r,.)]\,dr\partial_{z} \phi \,dx\,dz - \int_0^1\int_0^z U_{\text{inflow}}(r,\cdot)\,dr\partial_z\phi(0,z)\,dz \nonumber\\&-\int_0^1 U_{\text{inflow}}(r,\cdot)\phi(0,r)\,dr.
\label{eq:star}
\end{align}
Similar to the proof of claim \eqref{eq:claim1}, we can show
\begin{align*}
 \int_0^1\int_0^z U_{\text{inflow}}(r,\cdot)\,dr\partial_z\phi(0,r)\,dr = - \int_0^1 U_{\text{inflow}}(r,\cdot) \phi(0,r)\,dr.
\end{align*}
Thus, equation \eqref{eq:star} reduces to
\begin{align}
\int_{\Omega} Q\partial_z\phi\,dx\,dz =-\int_{\Omega} \partial_x \int_0^z U[S(.,r,.)]\,dr\partial_{z} \phi \,dx\,dz,
  \label{eq:limit-Q}
\end{align}
for any test function $\phi\in H^1(\Omega)$. Hence, the vertical velocity $Q$ satisfies
\begin{align*}
  Q = -\partial_x \int_0^z U[S(.,r,.)]\,dr.
\end{align*}
Equations \eqref{eq:limit-U} and \eqref{eq:limit-Q} show that the limit velocity components $U$ and $Q$ are nonlinear nonlocal operators of the limit saturation $S$ together with the horizontal velocity at the inflow boundary. Consequently, equations \eqref{eq:weakmodellimit}, \eqref{eq:limit-U} and \eqref{eq:limit-Q} imply that the limit $(S,U,Q)$ of the sequence of weak solutions $(S^{\gamma},p^{\gamma},U^{\gamma},Q^{\gamma},)_{\gamma>0}$ for the BTP-model \eqref{eq:BTP} satisfies Definition \ref{def:BVE-weaksolution} and is, therefore, a weak solutions for BVE-model \eqref{eq:BVE}, \eqref{eq:velocity}. 

\end{proof}
\section{Numerical Example}
\label{sec:numerical}
In this section, we present a numerical example that shows the convergence of numerical solutions for the dimensionless BTP-model \eqref{eq:BTP} to numerical solutions for the reduced BVE-model \eqref{eq:BVE}, \eqref{eq:velocity} as the geometrical parameter $\gamma$ reduces. We consider the dimensionless BTP-model \eqref{eq:BTP} with the fractional flow function 
\begin{align}
 f(S)= \frac{MS^2}{MS^2+(1-S)^2},
 \label{eq:diffusion}
\end{align}
where $M$ is the viscosity ratio of the defending phase and the invading phase. The model is also assumed to be satisfied in the domains $\Omega_{\gamma}=(0,L)\times(0,H)$ with decreasing geometrical parameter $\gamma\in\{1,1/5,1/25,1/125\}$, such that the domains' length is fixed $L=5$ and the widths are decreasing $H\in\{5,1,1/5,1/15,1/25\}$.

The initial and boundary conditions are given as
\begin{align}
\begin{array}{rll}
 S^{\gamma}(\cdot,\cdot,0)&=S_{0} &\text{ in } \Omega, \\
 S^{\gamma}&=S_{\text{inflow}} &\text{ on } \{0\}\times (0,1)\times [0,T],\\
 p^{\gamma}&=1 &\text{ on } \{0\}\times (0,1)\times [0,T],\\
 p^{\gamma}&=0 &\text{ on } \{1\}\times (0,1)\times [0,T],\\
 W^{\gamma}&= 0  &\text{ on }(0,1)\times\{0,1\}\times [0,T].
\end{array}
\label{eq:IBC}
\end{align}
In the following examples we choose the initial condition
\begin{align*}
 S_0(x,z)=g(x)S_{\text{inflow}}(z),
\end{align*}
where \begin{align}
 g(x)=\dfrac{(1-x)^2}{10^5x^2+(1-x)^2}\quad\text{ and }\quad S_{\text{inflow}}(z)=\left\{ 
\begin{array}{c l l}
	0 \quad &: & z\leq \frac{3}{10} \text{ and } z>\frac{7}{10},\\
	0.9 \quad &: &\frac{3}{10}<z\leq \frac{7}{10}.
\end{array} \right.
\label{eq:inflow}
 \end{align}

We discretize the dimensionless BTP-model and the nonlocal BVE-model \eqref{eq:BVE}, \eqref{eq:velocity} by applying mass-conservative finite-volume schemes as described in \cite{Armiti-Juber2018}. The schemes are based on Cartesian grids with number of vertical cells $N_z$ significantly less than that in the horizontal direction $N_x$ that fits to the case of thin domains. In the following example, we use a grid of $1000\times 100$ elements, viscosity ratio $M=2$, end time $T=0.3$ and we set $\hat{U}_{\text{inflow}}=1$ in equation \eqref{eq:velocity}.

\begin{figure}
\centering
\subfigure[BTP-model $\gamma=1$]{
\includegraphics[scale=0.18]{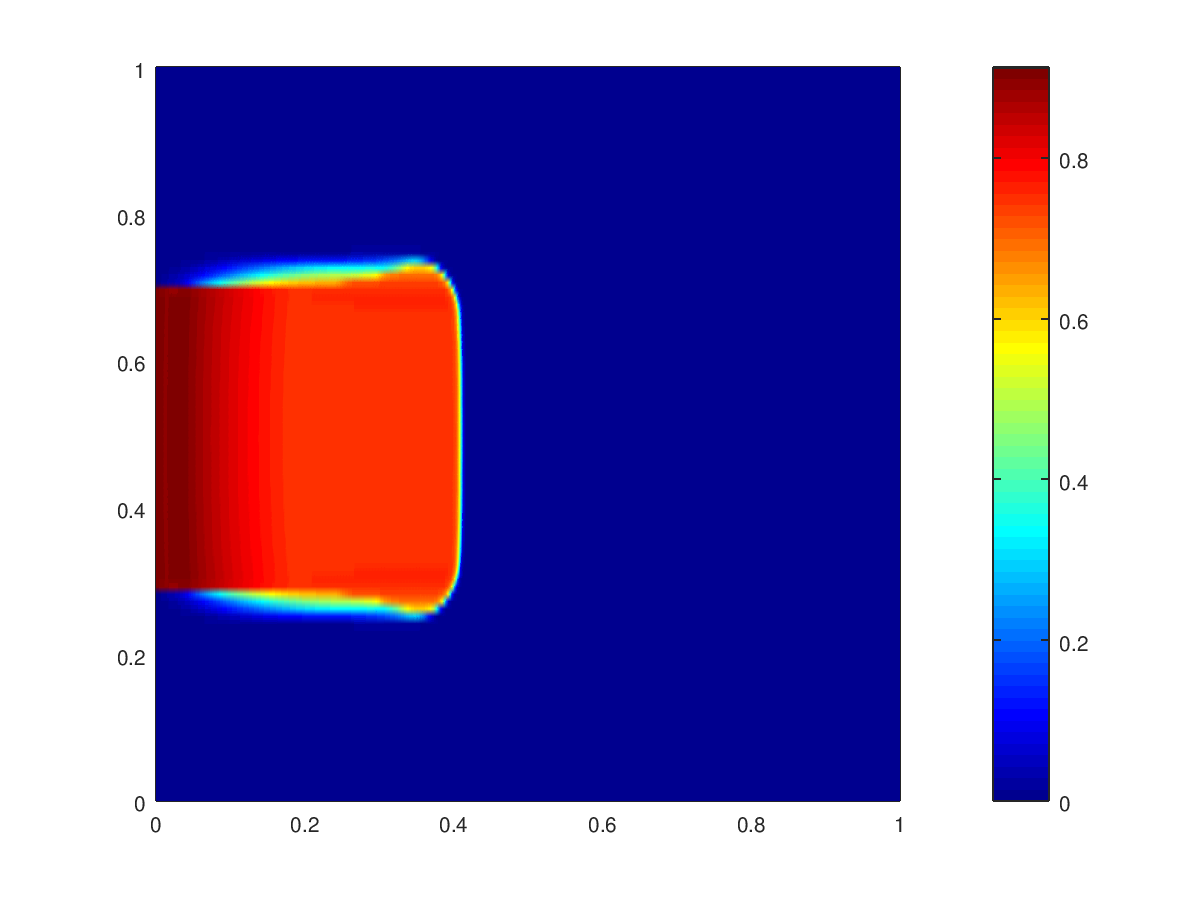}
\label{subfig:beta1}
}\hspace*{-0.1cm}
\subfigure[BTP-model $\gamma=1/5$]{
\includegraphics[scale=0.18]{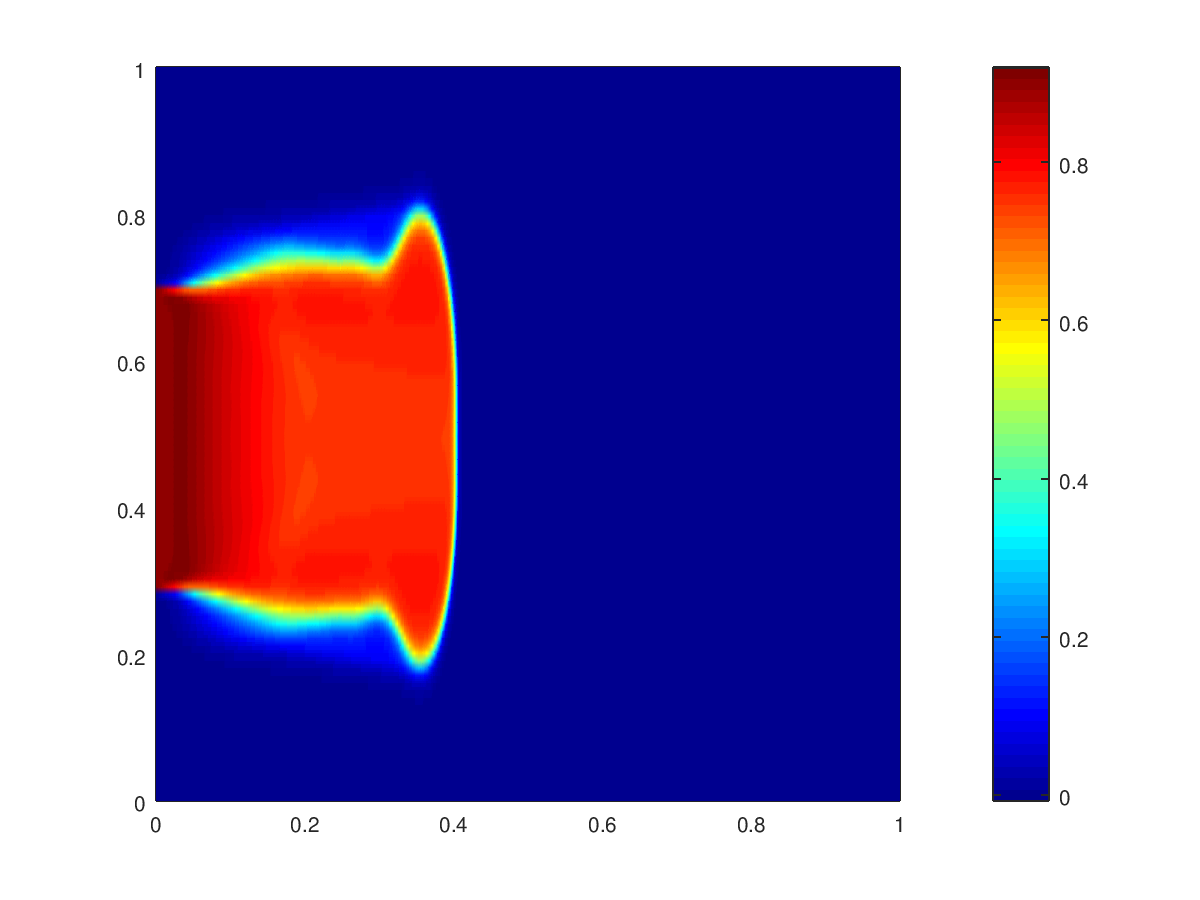}
\label{subfig:beta2}
}\hspace{-0.1cm}
\subfigure[BTP-model $\gamma=1/25$]{
\includegraphics[scale=0.18]{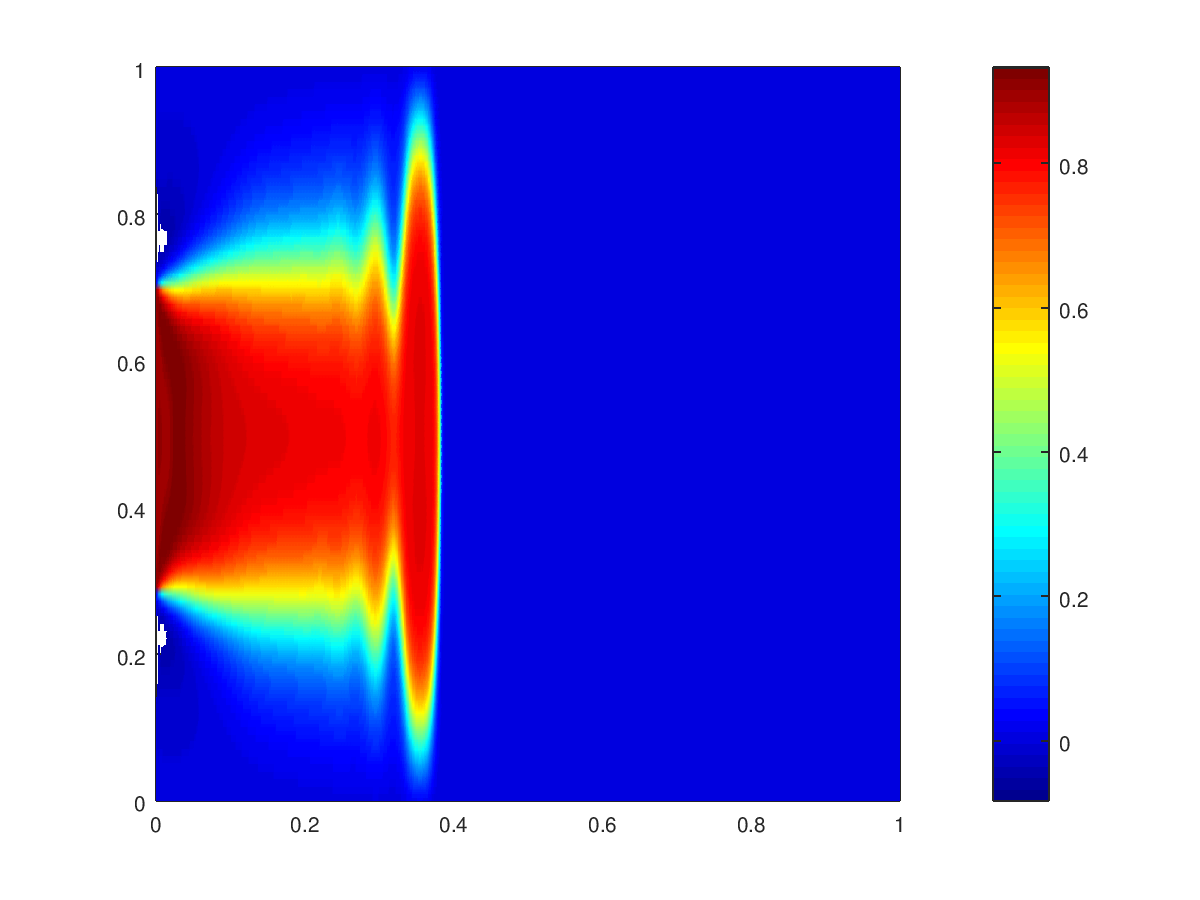}
\label{subfig:beta3}
}\\
\subfigure[BTP-model $\gamma=1/75$]{
\includegraphics[scale=0.18]{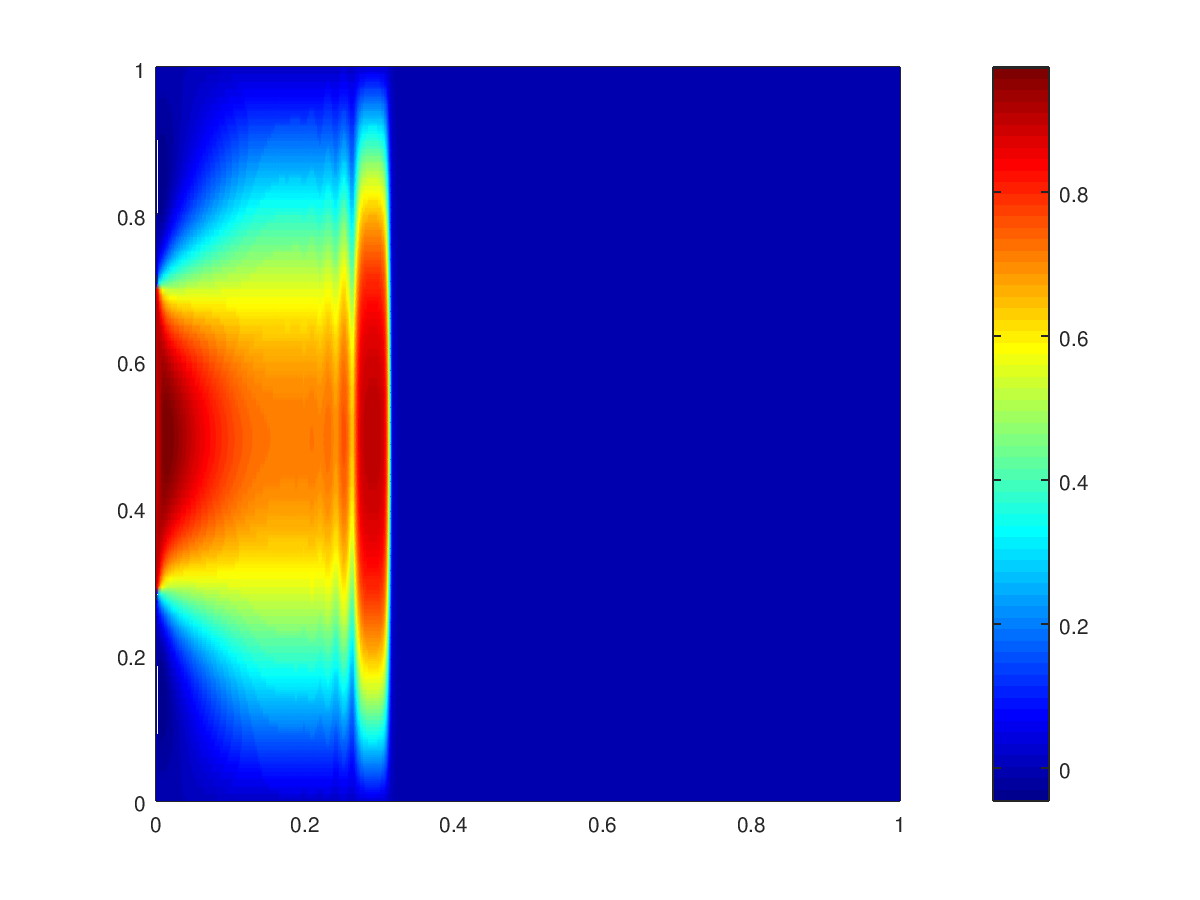}
\label{subfig:beta4}
}\hspace{-0.1cm}
\subfigure[BTP-model $\gamma=1/125$]{
\includegraphics[scale=0.18]{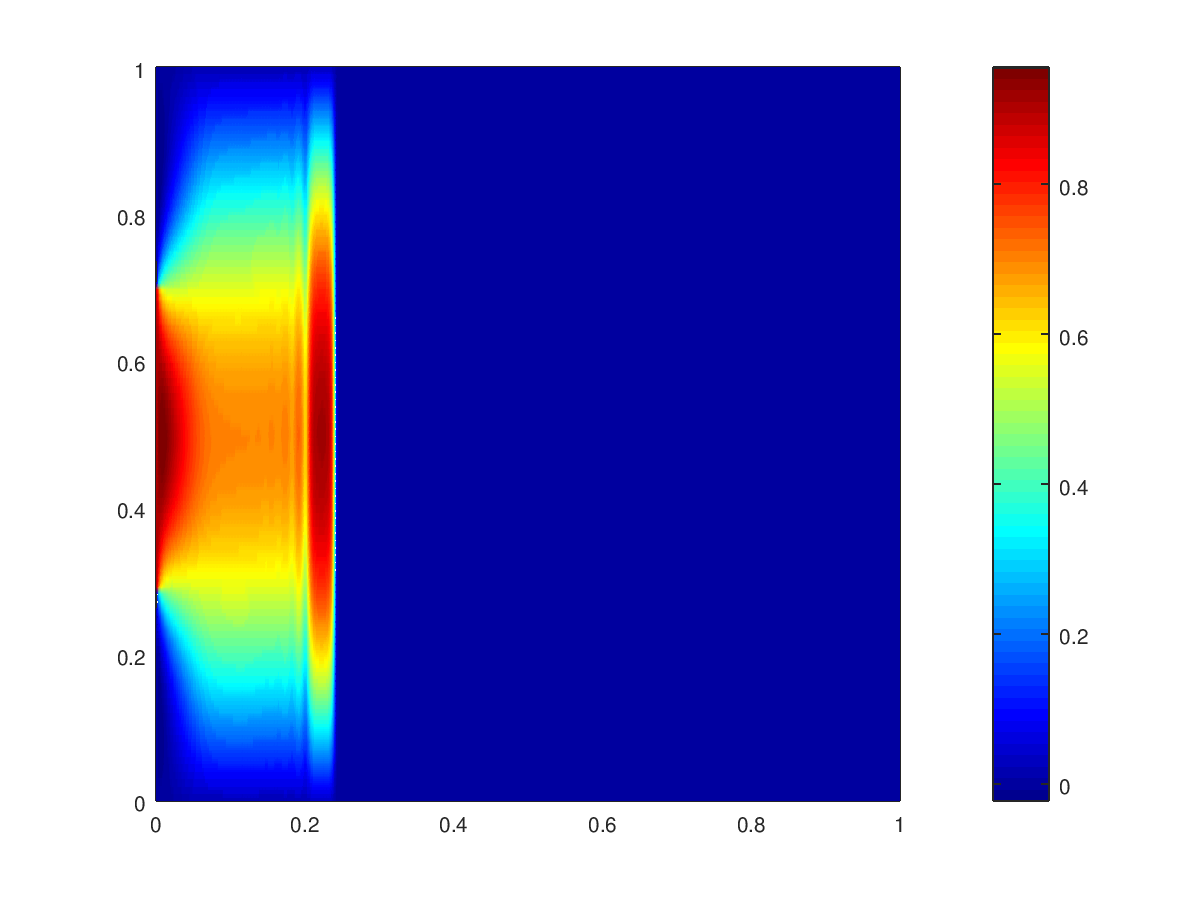}
\label{subfig:beta5}
}\hspace{-0.1cm}
\subfigure[BVE-model]{
\includegraphics[scale=0.18]{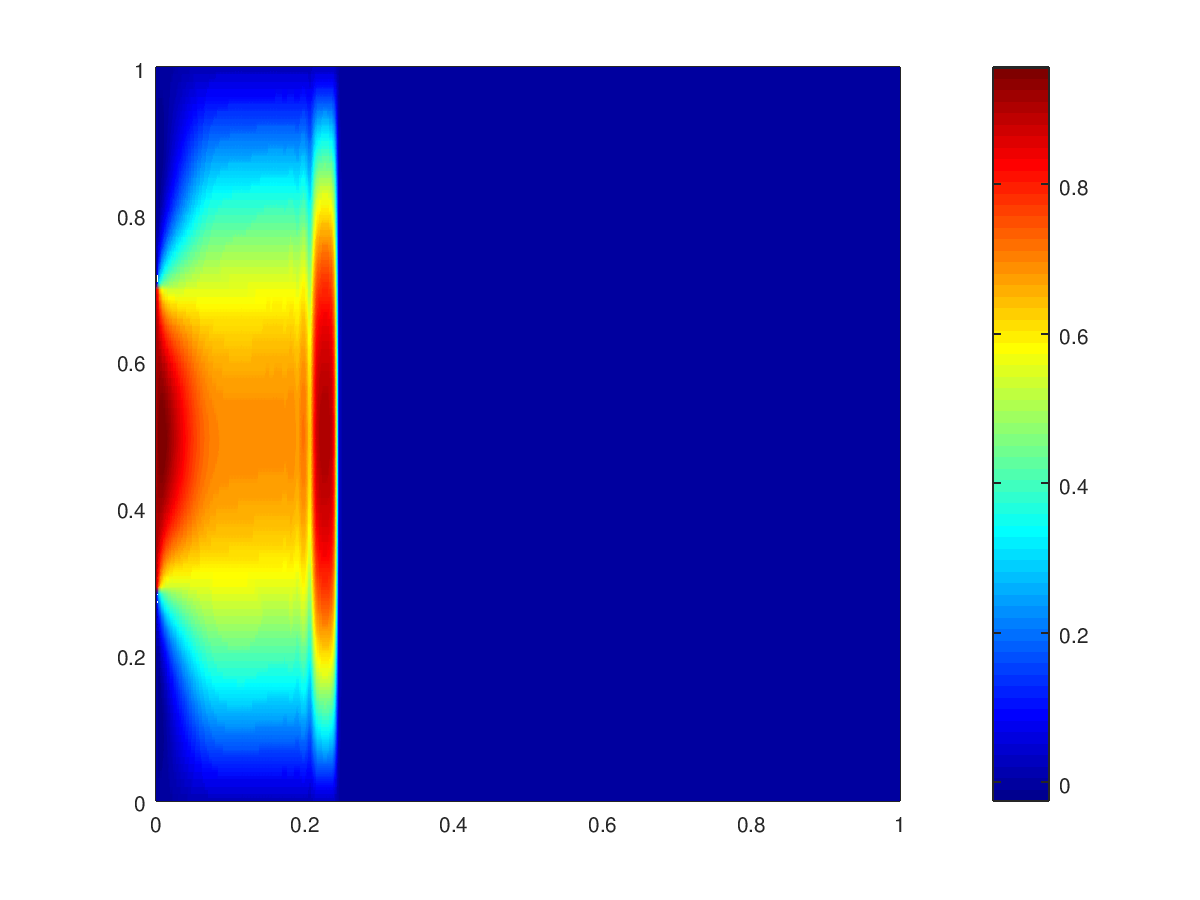}
\label{subfig:beta6}
}
\caption{Numerical solutions for the BTP-model \eqref{eq:BTP} in figures (a)-(e), with decreasing parameter $\gamma\in\{1,1/5,1/25,1/75,1/125\}$, converge to numerical solution for the BVE-model \eqref{eq:BVE}, \eqref{eq:velocity} in figure (f), using a $1000\times100$ grid, $M=2$, $\mu_e=10^{-2}$ and $T=0.3$.}
\label{fig:comparison1}
\end{figure}

In Figures \ref{subfig:beta1}-\ref{subfig:beta5}, we present the numerical solutions of the BTP-model \eqref{eq:BTP} using the parameters $\gamma\in\{1,1/5,1/25,1/75,1/125\}$, respectively, such that $L=5$ and $H\in\{5,1,1/5,1/15,1/25\}$. Figure \ref{subfig:beta6} presents the numerical solution of the BVE-model in the limit case with $L=5$ and $H=1/25$. The results in Figure \ref{fig:comparison1} suggest that numerical solutions for the BTP-model \eqref{eq:BTP} converge to the corresponding numerical solutions for the reduced BVE-model \eqref{eq:BVE}, \eqref{eq:velocity} as the geometrical parameter $\gamma$ tends to zero. This numerical convergence supports the theoretical results in Theorem \ref{thm:MainThm}.

\section{Conclusion}
\label{sec:conclusion}
We studied the limit of the two-phase flow model in porous media domains of Brinkman-type as the the domain's width--length ratio vanishes. We proved that weak solutions for this model converge to a weak limit. Further, we showed that the limit satisfies the definition of weak solutions for a model, in which pressure gradient is formulated as a nonlocal operator of saturation. 

The nonlocal model was first suggested in \cite{Armiti-Juber2018}  as a proper reduction of the full two-phase flow model in thin domains. It was derived using standard asymptotic analysis. However, the convergence analysis in this paper contributes to this model with a first rigid mathematical derivation.

\appendix
\section{The Dimensionless BTP-Model}
  \label{sec:dimensionless}
We consider the displacement process of two incompressible immiscible fluids in a saturated nondeformable porous medium of Brinkman type. The invading phase $\alpha=i$ is displacing the defending phase $\alpha=d$ under the assumption of negligible gravity and capillary forces. Then, the two-phase flow model consists of the continuity equation, the Brinkman equations and the incompressibility equation 
\begin{align}
\begin{array}{rl}
 \partial_{t} S_{\alpha} + \nabla\cdot \textbf{v}_{\alpha}&=0, \\
 - \mu_e \textbf{v}_{\alpha} + \textbf{v}_{\alpha}&=-\lambda_{\alpha}(S_{\alpha}) \textbf{K} \nabla p_{\alpha},\\
   \nabla\cdot \textbf{v}&=0
\end{array}
\label{eq:Brinkman-two-phase}
\end{align}
in $\Omega_{\gamma}\times (0,T)$, where $\Omega_{\gamma}=(0,L)\times(0,H)$ is a rectangular domain with the parameter $\gamma=H/L$. The intrinsic permeability tensor $\textbf{K} = \textbf{K}(x,z)$ is defined as $
 \textbf{K}(x,z) = \left(\begin{array}{c c}
               K_x(x,z) & 0 \\
               0 & K_z(x,z)
              \end{array}\right).$
We define the vector of generalized velocities
\begin{align}
 \textbf{V}_{\alpha} = - \mu_e \textbf{v}_{\alpha} + \textbf{v}_{\alpha},
 \label{eq:newdef}
\end{align}
such that $\textbf{V}_{\alpha}=(U_{\alpha},W_{\alpha})^T$. For this vector, we also define $\textbf{V}=\textbf{V}_{i}+\textbf{V}_{d}$, which satisfies the incompressibility-like equation
\begin{align}
 \nabla\cdot \textbf{V} = -\mu_e\,\Delta (\nabla\cdot \textbf{v}) +\nabla \cdot \textbf{v}=0.
\label{eq:incompressibilty2}
\end{align}

To derive the dimensionless BVE-model \eqref{eq:BTP} we rescale equation \eqref{eq:Brinkman-two-phase} using the dimensionless variables 
\begin{equation}
  \begin{array}{rlrlrlrl}
\overline{x}=&\dfrac{x}{L},\quad&\overline{z}=&\dfrac{z}{H},\quad & \overline{t}=&\dfrac{t}{L/q}, \quad& \kappa_{j}=&\dfrac{K_{j}}{k_{j}}\\
\overline{u}_{\alpha}=&\dfrac{u_{\alpha}}{q},\quad &\overline{w}_{\alpha}=&\dfrac{w_{\alpha}}{q},\quad & \overline{p}=&\dfrac{p}{Lq\mu_{d}/k_{x}},  
   \end{array}
 \label{eq:dimensionlessvariables}
\end{equation}
for $j\in\{x,z\}$ and $\alpha\in\{i,d\}$. Here, $q>0$ is the inflow speed at the inflow boundary $\partial\Omega_{\text{inflow}}$, $\mu_d$ is the viscosity of the defending phase and $k_j$ is the mean value of the corresponding permeability function $K_j$.
Applying the chain rule to \eqref{eq:newdef}, then defining the dimensionless components
\begin{align}
 \overline{U}_{\alpha}\coloneqq\dfrac{U_{\alpha}}{q}, \quad\quad \overline{W}_{\alpha}\coloneqq\dfrac{W_{\alpha}}{q},
 \label{eq:B-dime-compo}
\end{align}
yield
\begin{align}
\begin{array}{ll}
 \overline{U}_{\alpha}=&\overline{u}_{\alpha}-\frac{\mu_e}{L^2}\partial_{\overline{x}\overline{x}}\overline{u}_{\alpha}  -\frac{\mu_e}{H^2}\partial_{\overline{z}\overline{z}}\overline{u}_{\alpha},\vspace{0.1cm}\\
 \overline{W}_{\alpha}=&\overline{w}_{\alpha}-\frac{\mu_e}{L^2}\partial_{\overline{x}\overline{x}}\overline{w}_{\alpha}  -\frac{\mu_e}{H^2}\partial_{\overline{z}\overline{z}}\overline{w}_{\alpha}.
\end{array}
 \label{eq:operatorNon}
\end{align}
Applying the chain rule to equations \eqref{eq:Brinkman-two-phase} and \eqref{eq:incompressibilty2}, using equations \eqref{eq:dimensionlessvariables} and \eqref{eq:B-dime-compo}, then omitting the bar-signs leads to
\begin{align}
\begin{array}{rl}
 \partial_{t}S_{\alpha} +\partial_{x}u_{\alpha}+(1/\gamma) \partial_{z}w_{\alpha}&=0,\vspace{3pt}\\
 U_{\alpha}&=-\lambda_{\alpha}(S_{\alpha})\, \kappa_x\, \partial_{x} p_{\alpha},\vspace{3pt}\\
 (\gamma/\sigma) W_{\alpha}&=-\lambda_{\alpha}(S_{\alpha})\, \kappa_z\, \partial_{ z}  p_{\alpha},\vspace{3pt}\\
  \partial_{x}u+(1/\gamma) \partial_{z}w &=0,\vspace{3pt}\\
   \partial_{x}U+(1/\gamma) \partial_{z}W &=0
\end{array}
\label{eq:B-Dimensionlesstwo-phase2}
\end{align}
in $\Omega\times(0,T)$, for both invading and defending phases $\alpha\in\{i,d\}$ and $\sigma=k_z/k_x$. Now, applying the operator $1-\beta_1\partial_{xx}-\beta_2\partial_{zz}$ to the continuity equation, where $\beta_1=\frac{\mu_e}{L^2}$ and $\beta_2=\frac{\mu_e}{H^2}$, transforms model \eqref{eq:B-Dimensionlesstwo-phase2} to
\begin{align}
\begin{array}{rl}
 \partial_{t}S_{\alpha} -\beta_1\partial_{xxt}-\beta_2\partial_{zzt}S_{\alpha} + \partial_{x}U_{\alpha}+(1/\gamma) \partial_{z}W_{\alpha}&=0,\vspace{3pt}\\
 U_{\alpha}&=-\lambda_{\alpha}(S_{\alpha})\, \kappa_x\, \partial_{x} p_{\alpha},\vspace{3pt}\\
 (\gamma/\sigma) W_{\alpha}&=-\lambda_{\alpha}(S_{\alpha})\, \kappa_z\, \partial_{ z}  p_{\alpha},\vspace{3pt}\\
 \partial_{x}U+(1/\gamma) \partial_{z}W &=0.
\end{array}
\label{eq:B-Dimensionlesstwo-phase-22}
\end{align}

The assumption of negligible capillary pressure implies $p_i=p_d=:p$ and the phases' velocities satisfy 
\begin{align}
  U_{\alpha}=f(S_{\alpha})U,\quad\quad  W_{\alpha}=f(S_{\alpha})W.
  \label{eq:Dimensionlessfrac}
\end{align}
We also set $\kappa=1$ to simplify the analysis in this paper, and we define the variable $Q=W/\gamma$. Then, the dimensionless model \eqref{eq:B-Dimensionlesstwo-phase2} is summarized such that the unknown variables $S,\,p,\,U$, and $Q$ are associated with the parameter $\gamma$,
\begin{align}
\begin{array}{rl}
\partial_{t}S^{\gamma}-\beta_1\partial_{xxt} S-\beta_2\partial_{zzt}S&+\partial_{x}\left(f(S^{\gamma})U^{\gamma}\right)+ \partial_{z}\left(f(S^{\gamma})Q^{\gamma}\right)=0,\\
U^{\gamma}&=-\lambda_{tot}(S^{\gamma})  \partial_{x} p^{\gamma},\\
 \gamma^2 Q^{\gamma}&=-\lambda_{tot}(S^{\gamma})  \partial_{z} p^{\gamma},\\
  \partial_{x}U^{\gamma}+ \partial_{z}Q^{\gamma} &=0,
\end{array}
\label{eq:B-Dimensionlessfractionalformulation}
\end{align}
where $S=S_i$ is the saturation of the invading fluid.

\section{Asymptotic Analysis}
\label{sec:asymptotic}
The BVE-model is derived in \cite{Armiti-Juber2018} by applying formal asymptotic analysis, with respect to $\gamma$, to the dimensionless BTP-model \eqref{eq:B-Dimensionlessfractionalformulation}. We assume that each component in $(S^{\gamma},p^{\gamma},U^{\gamma}, Q^{\gamma})$ is smooth and can be written in terms of the asymptotic expansions
\begin{align}
 \begin{array}{ll}
  Z^{\gamma}=Z_{0}+\gamma Z_{1}+\mathcal{O}(\gamma^2),\,& Z^{\gamma}\in\{S^{\gamma},\,p^{\gamma},\,U^{\gamma}, Q^{\gamma}\}.
 \end{array}
 \label{eq:B-asymptoticexpansion}
\end{align}
Using the asymptotic expansion of $S^{\gamma}$ in \eqref{eq:B-asymptoticexpansion} and Assumption \ref{ass:Bexistence}, we have the Taylor expansions
\begin{align}
\begin{array}{cl}
G(S^{\gamma})&=G(S_{0})+ G'(S_{0})(\gamma S_{1})+\mathcal{O}(\gamma^{2}),
\end{array}
 \label{eq:B-asymptoticexpansion1}
\end{align}
for $G\in\{\lambda_{tot},\,f\}$. The incompressibility relation in \eqref{eq:B-Dimensionlessfractionalformulation} allows writing the continuity equation in nonconservative form. Substituting equation \eqref{eq:B-asymptoticexpansion} and \eqref{eq:B-asymptoticexpansion1} into \eqref{eq:B-Dimensionlessfractionalformulation}, the terms of order $\mathcal{O}(1)$ satisfy
\begin{align}
\begin{array}{rl}
 \partial_{t}S_{0}-\beta_1\partial_{xxt}S_{0}-\beta_2\partial_{zzt}S_{0}+\partial_{x}\big(f(S_{0})U_{0}\big)\vspace{5pt}+ \partial_{z}\big(f(S_{0})Q_0\big)&=\mathcal{O}(\gamma),\vspace{5pt}\\
  U_{0}&=-\lambda_{tot}(S_{0})  \partial_{x} p_{0},\vspace{5pt}\\
  \lambda_{tot}(S_{0}) \partial_{z} p_{0}&=\mathcal{O}(\gamma^2),\vspace{5pt}\\
 \partial_{x}U_{0}+\partial_{z}Q_{0} &=\mathcal{O}(\gamma).
\end{array}
\label{eq:B-asymptotic}
\end{align}
Using the positivity of the total mobility $\lambda_{tot}$ (see Assumptions \ref{ass:Bexistence}(6)), the third equation of \eqref{eq:B-asymptotic} implies that $p_0$ is independent of the $z$-coordinate,
\begin{equation}
 p_{0}=p_{0}(x,t).
 \label{eq:B-pressure-y-indep}
\end{equation}
Integrating the last equation in \eqref{eq:B-asymptotic} over the vertical direction from $0$ to $1$ and using the assumption of impermeable upper and lower boundaries of the domain $\partial_{\text{imp}}\Omega$ in \eqref{eq:IBC-BTP}, we obtain
\begin{eqnarray*}
\partial_{x}\int_{0}^{1}U_{0}\,dz =-\int_{0}^{1}\partial_{z}Q_{0}\,dz=0.
\end{eqnarray*}
Integrating this equation from $0$ to $x$ yields 
\begin{align}
  \int_{0}^{1}U_{0}(x,z,t)\,dz- h(t)=0,
  \label{eq:U1}
\end{align}
for any $x\in(0,1)$ and $t\in[0,T]$, where $h(t)=\int_{0}^{1}U_{0}(0,z,t)\,dz$ is the averaged horizontal velocity at the inflow boundary. Substituting the second equation in \eqref{eq:B-asymptotic} into equation \eqref{eq:U1} yields
\begin{equation*}
-\int_{0}^{1}\lambda_{tot}(S_{0}) \partial_{x} p_{0}\,dz=h(t).
\end{equation*}
Then, using equation \eqref{eq:B-pressure-y-indep}, we have
\begin{equation}
\partial_{x} p_{0}(x,t)=-\dfrac{h(t)}{\int_{0}^{1}\lambda_{tot}(S_{0}(x,z,t)) \,dz},
\label{eq:U2}
\end{equation}
for all $x\in (0,1)$ and $t\in(0,T)$. Substituting \eqref{eq:U2} into the second equation in \eqref{eq:B-asymptotic}, we obtain a nonlocal saturation-dependent formula for $U_0$,
\begin{equation}
 U_{0}[S_{0}]=\dfrac{h(t)\lambda_{tot}\bigl(S_{0}\bigr) }{\int_{0}^{1}\lambda_{tot}\bigl(S_{0}\bigr) \,dz},
  \label{eq:U3}
\end{equation}
for all $(x,z)\in \Omega$ and $t\in(0,T)$. Consequently, the incompressibility relation in \eqref{eq:B-asymptotic} yields also a nonlocal saturation-dependent formula for $Q_0$,
\begin{equation}
 Q_0[S_0]=-\partial_{x}\int_{0}^{z}U_0[S_0(\cdot,r,\cdot)]\,dr,
 \label{eq:U4}
\end{equation}
for all $(x,z)\in \Omega$ and $t\in(0,T)$.
Using equation \eqref{eq:U3} and \eqref{eq:U4}, omitting the subscript $\{0\}$, system \eqref{eq:B-asymptotic} reduces to a third-order nonlocal nonlinear equation of saturation
\begin{align}
 \partial_{t}S+\partial_{x}\left(f(S)U\right)+\partial_{z}\left(f(S)Q\right)-\beta_1\partial_{xxt}S-\beta_2\partial_{zzt}S_{0} =0,
\label{eq:Brinkman1}
\end{align}
in $\Omega\times (0,T)$ where we have for all $z\in(0,1)$
\begin{align}
\begin{array}{rl}
 U[S]&=\dfrac{\lambda_{tot}(S)}{\int_{0}^{1}\lambda_{tot}(S) \,dz}, \vspace{5pt}\\
 Q[S]&=-\partial_{x}\int_{0}^{z}U[S(\cdot,r,\cdot)]\,dr .
 \end{array}
 \label{eq:Brinkman2}
\end{align}
\medskip
\textbf{Acknowledgments. }
The author would like to thank Prof. Iuliu Sorin Pop, from Hasselt University, for the fruitful discussion and valuable comments.

%
%
%
%
%

\bibliographystyle{plain}
\bibliography{BTPtoBVE}

\end{document}